\documentclass[a4paper,10pt,intlimits]{amsart}
\newenvironment{keyword}{\noindent\textbf{Keywords: }}{

}

\usepackage{amsmath}
\usepackage{amsthm}

\usepackage{amstext}
\usepackage{amssymb}
\usepackage{graphicx} 
\usepackage{color}
\usepackage{epsfig}
\usepackage{pifont}
\usepackage{fancyhdr}
\usepackage{psfrag}
\usepackage{float}
\usepackage{longtable}
\usepackage{subfigure}
\usepackage[small,bf]{caption}

\usepackage{a4wide}

\allowdisplaybreaks[1]

\newtheorem{Definition}{Definition}[section]
\newtheorem{Lemma}[Definition]{Lemma}
\newtheorem{Remark}{Remark}
\newtheorem{Theorem}[Definition]{Theorem}
\newtheorem{Corollary}[Definition]{Corollary}

\newcommand{\rieszIso}{{\mathcal R}}				
\newcommand{\linFuncSpace}{\mathcal L}				

\newcommand{\ts}[1]{^{(#1)}}					
\newcommand{\vecb}[1]{\boldsymbol{#1}}				
\newcommand{\normvec}{\vecb n}					
\newcommand{\dx}{\,d\vecb x}					
\newcommand{\drx}{\,d\omega(\vecb x)}				

\newcommand{\inverseTrans}[1]{(\Phi\ts{#1})^{-1}}

\newcommand{\transUN}[2]{\Phi\ts{#1}\!\{#2\}}
\newcommand{\transUn}[1]{\transUN{n}{#1}}
\newcommand{\inverseTransUn}[1]{(\transUn{#1})^{-1}}

\renewcommand{\phi}{\varphi}
\newcommand{\eps}{\varepsilon}

\renewcommand{\rho}{\varrho}
\renewcommand{\epsilon}{\varepsilon}
\renewcommand{\theta}{\vartheta}

\let\Circ\circ
\renewcommand{\circ}{\!\Circ\!}
\newcommand{\<}{\langle}
\renewcommand{\>}{\rangle}

\DeclareMathOperator*{\Id}{Id}

\DeclareMathOperator*{\trace}{Tr}

\DeclareMathOperator*{\Div}{div}

\newcommand{\minCont}{\ensuremath{(\mathcal P)\,}}

\newcommand{\titleName}{Adjoint-Based Optimal Control of Time-Dependent Free Boundary Problems}
\title{\titleName}

\author{Jan Marburger$^*$}
\address[*]{Fraunhofer-Institut f\"ur Techno- und Wirtschaftsmathematik\\
	Fraunhofer-Platz 1 \\
	D-67663 Kaiserslautern \\
	e-mail: jan@nit-service.de}

\begin{document}
\bibliographystyle{plain} 

\begin{abstract}
In this paper we show a simplified optimisation approach for free boundary problems in arbitrary space dimensions.
This approach is mainly based on an extended operator splitting which allows a decoupling of the domain deformation and solving the remaining partial differential equation.
First we give a short introduction to free boundary problems and the problems occurring in optimisation. 
Then we introduce the extended operator splitting and apply it to a general minimisation subject to a time-dependent scalar-valued partial differential equation.
This yields a time-discretised optimisation problem which allows us a quite simple application of adjoint-based optimisation methods.
Finally, we verify this approach numerically by the optimisation of a flow problem (Navier-Stokes equation) and the final shape of a Stefan-type problem.
\end{abstract}

\maketitle


\pagestyle{myheadings}
\thispagestyle{empty}
\markboth{Jan Marburger}{\titleName}

\section{Introduction}
Optimisation of free surface problems \cite{colli2004free,shyy2007computational} often occurs in industrial applications. Some examples are the stabilisation of a liquid surface for sloshing \cite{ibrahim2005liquid} or optimising the shape of solidification processes \cite{hinze2007control}.
Free surface problems are still challenging from an analytical as well as a numerical point of view. Here, the domain is an unknown of the equation system which depends on the states, e.g.\ a water surface is driven by the flow velocity.
Since these problem are already hard to handle, the optimisation of such processes is very complex. Especially adjoint-based approaches \cite{Troel} are very difficult to apply due to the state-dependent domain.
Here, several assumptions and methods were derived to handle this kind of problem. For special cases it is possible to describe the free boundary by a graph \cite{dissSabine} or introducing a level-set or phase field function \cite{bernauer2010optimal}. Another approach is the pullback of the time- and state-dependent domain to a reference domain and perform all calculations in there. From an optimisation point of view, all of these methods have the disadvantage of very complex derivatives describing the variation of the domain. Note that often these derivatives are, in contrast to stationary problems, hard to interpret for time-dependent problems.

In this paper we show a simplified optimisation approach for free boundary problems in arbitrary space dimensions which bases on an extended time-discretisation of the problem.
We consider the problem: Minimise $J(y,\Omega,u)$ subject to the free surface problem of finding $(y,\Omega)$ such that
\begin{align}
  \begin{aligned}
  \partial_t y(t) + A(t)y(t) &=  f(u(t))  &&\mbox{in }\Omega(t) \\
  B(t)y(t) &= g(u(t))          &&\mbox{on }\Gamma(t)
  \end{aligned}
  \hspace{2cm}
  \begin{aligned}
   \begin{aligned}
     y(0) &= y_0  \quad\mbox{in }\Omega(0) \\
     &C(y,\Omega) = 0
   \end{aligned}
 \end{aligned}
 \label{equ:domTimeDep}
\end{align}
holds for all $t\in (0,T]$.
Here, $\Omega(t)\subset\mathbb R^d$ denotes the time- and state-dependent domain, $y$ a scalar-valued function,  $A$ an arbitrary differential operator and $C(y,\Omega)$ a constraint function defining the free boundary.
Moreover, $B$ and $g:\mathbb R\to\mathbb R$ denote appropriate boundary conditions, $u$ a control function and $f:\mathbb R\to\mathbb R$ a right hand side term depending on $u$, e.g.\ a localisation function. The difficulty is the dependency of the domain $\Omega(t)$ on the solution $y$ of the partial differential equation.
To obtain the domain, we solve, roughly speaking, a minimisation problem for each time step $t\in(0,T]$ in order to fulfil the constraint for the free boundary.
Hence, the minimisation of the cost functional $J$ would be subjected to the minimisation of the constraint function $C$ in order to solve the state equation 
\eqref{equ:domTimeDep}.

\begin{figure}
 \centering
 \begin{minipage}[t]{70mm}
  \includegraphics[width=.98\textwidth]{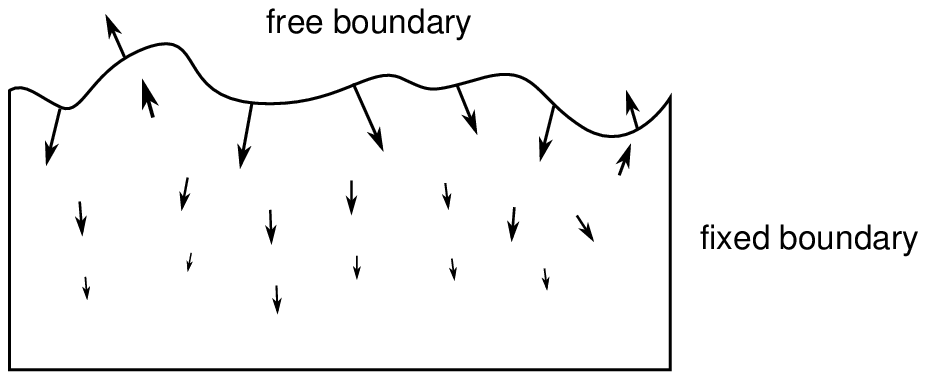}
  \caption{Artificial convection. The arrows illustrate the flux function $F$ which provides the deformation field to generate the new domain or 
  transformation $\Phi$.}
  \label{fig:convection}
 \end{minipage}
 \hspace*{10mm}
 \begin{minipage}[t]{70mm}
  \includegraphics[width=.98\textwidth]{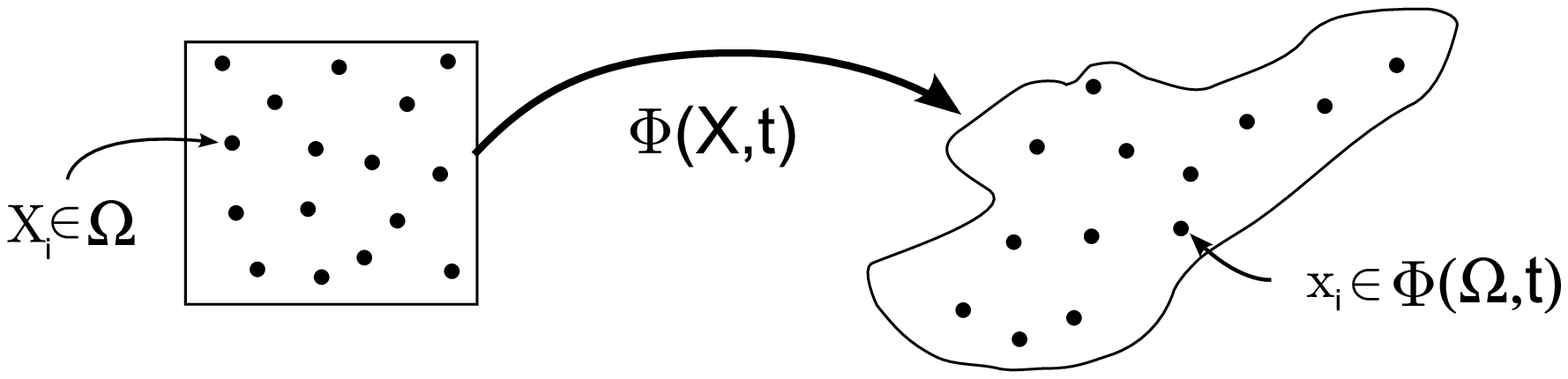}
  \caption{Pullback to a reference domain by the transformation $\Phi$. Capital letters, e.g.\ $X$ denote coordinates in the reference domain (left)
  		   and lower case letters, e.g.\ $x$, denote coordinates in the transformed domain.}
  \label{fig:transformation}
 \end{minipage}
\end{figure}

In order to avoid the minimisation problem for solving the state equation, we perform a complete pullback of equation \eqref{equ:domTimeDep} to a reference domain $\hat\Omega$. For this, we introduce a flux function $F:\mathbb R\to \mathbb R^d$ depending on the state $y$. This function is given by a characteristic velocity (e.g.\ flow) or a smooth continuation of the boundary motion, cf.\ figure~\ref{fig:convection}. Together with the transformation $\Phi$ given by
\begin{align*}
 \partial_t \Phi = F(y\circ\Phi) \quad\mbox{in }\hat\Omega\times(0,T)
 \qquad\mbox{ with }\qquad
 \Phi(X,0) = X \quad\mbox{in }\hat\Omega
\end{align*}
we resolve the constraint function for the domain by $\Omega(t)=\Phi(\hat\Omega,t)$ and hence $C(y,\Phi(\hat\Omega)) = 0$ holds, see figure~\ref{fig:transformation}. 
Therefore, we reformulate the original minimisation problem of the cost function $J$ to: Minimise $\hat J(\hat y,\Phi, u)$ subject to
\begin{align}
 \begin{aligned}
  \partial_t \hat y + \hat A(\Phi) \hat y &= f(u\circ\Phi)  && \mbox{in } \hat\Omega \times (0,T) \\
  \hat B(\Phi) \hat y &= g(u\circ\Phi)  && \mbox{in } \hat\Gamma \times (0,T)\\
  \hat y(0) &= y_0  &&\mbox{in }\hat\Omega
 \end{aligned}
 \hspace{2cm}
 \begin{aligned}
  \partial_t \Phi &= F(\hat y)  &&\mbox{in }\hat\Omega \times (0,T) \\
  \Phi(0) &= \Id{}_{\hat\Omega} &&\mbox{in }\hat\Omega \\
 \end{aligned}
 \label{equ:domPullback}
\end{align}
with $\hat y :=\! y\Circ\Phi$ and $\Id{}_{\hat\Omega}(x):=x$ as the identity map in $\hat\Omega$. Note that also the operators $A$ and $B$ change to $\hat A(\Phi)$ and $\hat B(\Phi)$, respectively.
Now the domain is fixed but the dependency of the differential operator on $\Phi$ yields very complex derivatives.

In the following we derive a simplified optimisation approach for free boundary problems using an extended time discretisation which is based on an operator splitting.
This approach is a blend of the original \eqref{equ:domTimeDep} and transformed formulation \eqref{equ:domPullback} of the optimisation problem.


\section{Formulation of the State Equation}
In this section we reformulate the state equation by a time-discretisation. Particularly this is done by applying an operator splitting scheme, which allows a simplified optimisation approach later on.
First, we show a first order splitting scheme for a simple convection-diffusion equation in $\mathbb R^{d}$ and transfer the results to the state equation \eqref{equ:domTimeDep}. Finally, we introduce an appropriate time-discrete Hilbert space which simplifies the application of adjoint-based optimisation.
\subsection{Basics of Operator Splitting}
In the following we use a Yanenko splitting \cite{splittingMethods} which is illustrated by a simple convection-diffusion equation in $\mathbb R^{d}$
\begin{align*}
 \partial_t y + \vecb v \cdot\nabla y - \Delta y &= f
 \quad\mbox{in }\mathbb R^d\times(0,T]
 &\mbox{with}&&
 y(0) &= y_0 \quad\mbox{in }\mathbb R^d
\end{align*}
where $y:\mathbb R^{d}\to \mathbb R$ and $\vecb v:\mathbb R^d\to\mathbb R^d$ denotes a smooth vector field. Note that the velocity field can also depend on $y$. This problem is divided into two subproblems. For the first time interval $[0,\tau]$ the Yanenko splitting reads
\begin{align*}
 \begin{aligned}
  \partial_t y^* + \vecb v\cdot\nabla y^* &= 0 
  &&\mbox{in }\mathbb R^d\times(0,\tau] \\
  y^*(0) &= y_0
  &&\mbox{in }\mathbb R^d
 \end{aligned}
 \quad\rightarrow\quad
 \begin{aligned}
  \partial_t y - \Delta y &= f  &&\mbox{in }\mathbb R^d\times(0,\tau] \\
  y(0) &= y^*(\tau)  &&\mbox{in }\mathbb R^d
 \end{aligned}
\end{align*}
The resulting subproblems can be solved by different methods. 
In particular, we consider the convection part from a Lagrangian viewpoint, cf.\ \cite{GODL1996}, which yields
\begin{align*}
 \begin{aligned}
  \partial_t \hat y^* &= 0		&&\mbox{in } \mathbb R^d\times(0,\tau] \\
  \hat y^*(0) &= y_0 			&&\mbox{in } \mathbb R^d
 \end{aligned}
 \hspace*{2cm}
 \begin{aligned}
  \partial_t \Phi &= \vecb v\circ\Phi	&&\mbox{in }\mathbb R^d \times(0,\tau] \\
  \Phi(X,0) &= X			&&\mbox{in }\mathbb R^d
 \end{aligned}
\end{align*}
for $\hat y^* := y^*\circ \Phi$. Thus the solution $y^*$ at time $\tau$ is given by
\begin{align*}
 y^*(\Phi(X,\tau),\tau) = \hat y^*(X,\tau) &= y_0(X) 
\end{align*}
and hence $y^*(\tau) = y_0(\Phi^{-1}(\tau))$ where $\Phi^{-1}$ denotes the inverse of $\Phi$.
The diffusion part is written as
\begin{align*}
 \partial_t y - \Delta y &= f		&&\mbox{in }\Phi(\mathbb R^d,\tau)\times(0,\tau] \\
 y(0) &= y_0(\Phi^{-1}(\tau))		&&\mbox{in }\Phi(\mathbb R^d,\tau)
\end{align*}
Note that $\Phi(\mathbb R^d,\tau) = \mathbb R^d$ but the metric, needed for example for spatial operators, is induced by the transformation $\Phi$.
Applying this scheme iteratively, we solve the convection-diffusion equation for all time intervalls. 
Finally, a time-discretisation is performed. Since the Yanenko splitting is of order $\mathcal O(\tau)$, a first order scheme is sufficient.
\begin{align*}
 \Phi\ts{n} &= \Id{}_{\Omega\ts{n}} + \tau \vecb v(t\ts{n})
 &&\mbox{in }\Omega\ts{n} \\
 \frac1\tau\big( y\ts{n+1} - y\ts{n}\circ(\Phi\ts{n})^{-1} \big) - \Delta y\ts{n+1} &= f(t\ts{n+1})
 &&\mbox{in }\Omega\ts{n+1}
\end{align*}
with $\Omega\ts{n+1} = \Phi\ts{n}(\Omega\ts{n})$.
This approach, i.e.\ solving the convection part by a Lagrangian viewpoint, is the main principle of particle methods, cf.\ \cite{MyPhd}, which will be used for the numerical results later on.

\begin{Remark}
 Due to the definition, the continuous transformation $\Phi$, given by \eqref{equ:domPullback}, is a diffeomorphism, cf.\ \cite{Raviart2}. For small time steps also the time-discrete transformations $\Phi\ts{n}$ are diffeomorphisms. 
 This implies that $\det(\nabla\Phi\ts{n}) > 0$ is always satisfied.
\end{Remark}


\subsection{Time Discretisation of the State Equation}
\label{sec:timeDiscr}
Now we apply the above splitting scheme to the state equation \eqref{equ:domTimeDep}.
For this, we extend \eqref{equ:domTimeDep} by a convection part given by a smooth continuation of the boundary motion. Particularly, we obtain
\begin{align*}
  \begin{aligned}
   \partial_t y(t) + F(y(t))\cdot\nabla y(t) - F(y(t))\cdot\nabla y(t) + A(t) y(t) &= f(u(t)) \qquad &&\mbox{in }\Omega(t) \\
   B(t)y(t) &= g(u(t))        &&\mbox{on }\partial\Omega(t) \\
   y(0) &= y_0             &&\mbox{in }\Omega(0)
  \end{aligned}
\end{align*}
for all $t\in (0,T]$. 
Here the flux function $F:\mathbb R\to\mathbb R^d$ defines the deformation of the entire domain, cf. equation \eqref{equ:domPullback} and figure~\ref{fig:convection}.
Applying the splitting scheme of the previous section to this equation we obtain the time-discrete system
\begin{align}
 \label{equ:timediscrState}
 \begin{aligned}
 \Phi\ts{n} = \Id{}_{\Omega\ts{n}} &+ \tau F(y\ts{n})   &&\mbox{in }\overline{\Omega\ts{n}}\\
 \frac1\tau\big(y\ts{n+1} - y\ts{n}\circ(\Phi\ts{n})^{-1} \big) 
  - \big(F(y\ts{n})\cdot\nabla y\ts{n}\big)\circ(\Phi\ts{n})^{-1} + A(t\ts{n+1})y\ts{n+1} &= f(u\ts{n+1})
 &&\mbox{in }\Omega\ts{n+1}\\
 B(t\ts{n+1}) y\ts{n+1} &= g(u\ts{n+1})
 &&\mbox{on }\partial\Omega\ts{n+1}
 \end{aligned}
\end{align}
with $\Omega\ts{n+1} := \Omega(t\ts{n+1}) = \Phi\ts{n}(\Omega\ts{n})$.
Note that the transformation, which generates the new domain, is solved explicitly. For this reason it is sufficient to treat the artificial convection term $F(y)\cdot\nabla y$ also explicitly in the above equation. Moreover, the transformation $\Phi\ts{n}$ depends on $y\ts{n}$ only. Hence we treat $\Phi\ts{n}$ as function of $y\ts{n}$ given by
\begin{align}
 \label{equ:transUn}
 \transUn{y} := \Id{}_{\Omega\ts{n}} + \tau F(y\ts{n})
\end{align}
in the following.
To consider weak formulations later on, we introduce the function space
\begin{align}
 \label{equ:def:hilbertSpace}
 V := \prod_{n=0}^{N_t} V\ts{n}
\end{align}
where $V\ts{n}$ denotes the spatial space at time $t\ts{n}$, for instance $V\ts{n} := H^1(\Omega\ts{n})$. These spaces are built recursively, that is,
\begin{align*}
 V\ts{0} \xrightarrow{\Phi\ts{0}} V\ts{1} \xrightarrow{\Phi\ts{1}} V\ts{2} \xrightarrow{\Phi\ts{2}} \cdots
\end{align*}
and depend on the data, e.g.\ right hand side or boundary conditions. 
If $V\ts{n}$ are Hilbert spaces, we define the corresponding inner product by
\begin{align}
 \label{equ:VinnerProd}
 \< x,y \>_V := \sum_{n=0}^{N_t} \tau \< x\ts{n}, y\ts{n} \>_{V\ts{n}}
\end{align}
for $x,y\in V$.
Note that for small time steps every transformation is a diffeomorphism which maps from $\Omega\ts{n}$ to $\Omega\ts{n+1}$ as state in the above remark. Therefore, the domain does not become singular and hence if $V\ts{n}$ is a Hilbert space then also $V\ts{n+1}$ is a Hilbert space. Consequently, \eqref{equ:VinnerProd} is an inner product and hence also $V$ is a Hilbert space.


\section{Optimal Control Problem}
Let the state space $V$ and the space of Lagrange multipliers $Z$ be Hilbert spaces which are defined analogous to \eqref{equ:def:hilbertSpace}. Let the space of control functions $U$ be a Hilbert space.
Moreover, we introduce the space 
\begin{align*}
 H:=\prod_{n=0}^{N_t} H\ts{n}
 \qquad\mbox{ with }\qquad
 H\ts{n}:=L^2(\Omega\ts{n}).
\end{align*}
The index $\Gamma$ denotes the corresponding function space on the boundary, e.g.\ $H\ts{n}_\Gamma = L^2(\partial\Omega\ts{n})$.

\subsection{Weak Formulation}
We generalise equation \eqref{equ:timediscrState} in a weak sense using the function spaces defined above as
\begin{align}
 \label{equ:dualFormState}
 \frac1\tau\big(y\ts{n+1} - y\ts{n}\circ\inverseTransUn{y} \big) 
  - \big(F(y\ts{n})\cdot\nabla y\ts{n}\big)\circ\inverseTransUn{y} + A\ts{n+1}y\ts{n+1} &= \mathcal B\ts{n+1} u
\end{align}
in $(V\ts{n+1})^*$. The operator $A\ts{n}:V\ts{n}\to (V\ts{n})^*$ denotes the weak counterpart of $A$ and $B$ used in equation \eqref{equ:timediscrState} and 
$\mathcal B\ts{n}\in \linFuncSpace(U;(V\ts{n})^*)$ of $f$ and $g$.
For example, we obtain for the Laplacian $A\ts{n}y\ts{n}:=\Delta y\ts{n}$ with Neumann boundary, $f(u) = 0$ and $g(u) = u$
\begin{align*}
 \< A\ts{n}y\ts{n},\lambda\>_{(V\ts{n})^*,V\ts{n}} := -\int_{\Omega\ts{n}} \nabla y\ts{n}\cdot\nabla\lambda  \dx
 \qquad\mbox{and}\qquad
 \< \mathcal B\ts{n} u, \lambda\>_{(V\ts{n})^*,V\ts{n}} := \int_{\partial\Omega\ts{n}} u \lambda\drx
\end{align*}
for $\lambda\in V\ts{n}$.

\subsection{Minimisation Problem}
We consider the minimisation problem
\begin{flalign*}
 &\minCont\hspace*{3cm}
 \min_{(y,u)\in V\times U} J(y,u)
 \qquad\mbox{ subject to }\qquad
 e(y,u) = 0 &
\end{flalign*}
where $J: V\times U \to \mathbb R^+_0$ is a cost functional and $e:V\times U\to Z^*$ is determined by \eqref{equ:dualFormState}, i.e.\ 
\begin{align}
 \label{equ:econstraint}
 \begin{aligned} 
   \<e(y,u), &\lambda\>_{Z^*,Z} :=  
   \sum_{n=1}^{N_t} \Big[ \<y\ts{n},\lambda\ts{n}\>_{H\ts{n}} 
     + \tau \<A\ts{n}y\ts{n},\lambda\ts{n}\>_{(V\ts{n})^*,V\ts{n}} 
     - \tau \<\mathcal B\ts{n}u,\lambda\ts{n}\>_{(V\ts{n})^*,V\ts{n}} 
   \Big]\\
   &- \sum_{n=0}^{N_t-1} \Big[ \<y\ts{n}\circ\inverseTransUn{y},\lambda\ts{n+1}\>_{H\ts{n+1}} 
   + \tau \< (F(y\ts{n})\cdot\nabla y\ts{n})\circ\inverseTransUn{y},\lambda\ts{n+1} \>_{H\ts{n+1}}  \Big] \\
   &+ \<y\ts{0}-y_0,\lambda\ts{0}\>_{H\ts{0}} 
 \end{aligned}
\end{align}
Note that the constraint function $e$ does not depend on the domains $\Omega\ts{n}$ explicitly. The condition $\Omega\ts{n+1}=\Phi\ts{n}(\Omega\ts{n})$, needed to establish the space $V\ts{n+1}$, is given implicitly by the definition of $\Phi\ts{n}$.
\begin{Remark}
 On the one hand, the missing condition $\Omega\ts{n+1}=\Phi\ts{n}(\Omega\ts{n})$ in the above contraint function yields an underdetermined system. 
 On the other hand, the solution of the approach introduced in section \ref{sec:timeDiscr} satisfies the constraint function.
 This solution is used to establish the optimality condition and hence we assume the domains to be known.
 In standard approaches, fixing the domain would yield no information about the variation of the domain or the quantity discribing it, e.g.\ a tranformation.
 Here, this information is not fully lost as we still obtain information about domain variations by the push forward terms in the constraint and cost function.
\end{Remark}
The above minimisation problem is the time-discrete counterpart to the minimisation of $J(y,\Omega,u)$ stated in the very beginning. Here, all spatial dependencies, e.g.\ integration domains or evaluation positions, in the cost functional are replaced by the corresponding transformation $\transUn{y}$.
More details about the reformulation of the cost functional can be found in section \ref{sec:numRes} or \cite{MyPhd}.

To find a minimum of \minCont we use the Lagrangian multiplier theorem, that is, we determine the critical points of the Lagrange functional
$L:V\times U\times Z\to\mathbb R$ defined by
\begin{align*}
 L(y,u,\lambda) := J(y,u) + \<e(y,u),\lambda\>_{Z^*,Z}
\end{align*}
Then the Karush-Kuhn-Tucker system reads
\begin{align*}
 \<\partial_y L(y,u,\lambda), \psi_y\>_{V^*,V} = 0,
 \qquad
 \<\partial_u L(y,u,\lambda), \psi_u\>_{U^*,U} = 0
 \quad\mbox{and}\quad
 \<\partial_\lambda L(y,u,\lambda), \psi_\lambda\>_{Z^*,Z} = 0
\end{align*}
for all $\psi_y\in V$, $\psi_u\in U$ and $\psi_\lambda\in Z$. For more detail we refer to \cite{Troel}. The partial derivatives in $V$ and $Z$ are interpreted as
\begin{align*}
 \<\partial_y L(y,u,\lambda), \psi\>_{V^*,V} = \sum_{n=0}^{N_t} \tau \<\partial_{y\ts{n}} L(y,u,\lambda), \psi\ts{n}\>_{(V\ts{n})^*,V\ts{n}}
\end{align*}


\noindent The derivatives with respect to the control $u$ and the states $y\ts{n}$ are straight forward except for the push forward terms, whose derivatives are derived in the following.

\begin{Lemma}
 \label{lem:divDetGradPhi0}
 Let $\Omega\subset \mathbb R^d$ and $\Phi:\Omega\to\mathbb R^d$ be a smooth diffeomorphism. Then for $d=1,2,3$
 \begin{align*}
  \Div(\det(\nabla\Phi)\nabla\Phi^{-T}) = 0
 \end{align*}
 holds.
\end{Lemma}
\begin{proof}
 See \cite{MyPhd}, p. 27.
\end{proof}

\begin{Lemma}
 \label{lem:dPhiIntegralyLambda}
 Let $\Omega\subset\mathbb R^d$ and $\Phi$ be a smooth diffeomorphism. Moreover, let $y:\Omega\to\mathbb R$ and $\lambda:\Phi(\Omega)\to\mathbb R$ be sufficiently smooth. 
 Then
 \begin{align*}
  \partial_\Phi \Big( \int_{\Omega} y \,(\lambda\circ\Phi) \det(\nabla \Phi) \dx \Big) [\psi] =
  &\int_{\partial\Omega} y \,(\lambda\circ\Phi)\det(\nabla\Phi)\nabla\Phi^{-T} \normvec \cdot \psi \drx \\
   &- \int_\Omega  \det(\nabla\Phi) \nabla\Phi^{-T} \nabla y \,\lambda\circ\Phi \cdot\psi \dx
 \end{align*}
 holds for $\psi:\Omega\to\mathbb R^d$.
\end{Lemma}
\begin{proof}
 \begin{align*}
  &\partial_\Phi \Big(\! \int_{\Omega}\! y \,\lambda\circ\Phi \det(\nabla \Phi) \dx \Big) [\psi] =
   \int_\Omega y \,(D\lambda\circ\Phi)\psi\det(\nabla\Phi) + y\,\lambda\circ\Phi\det(\nabla\Phi) \nabla\Phi^{-T}:\nabla\psi \dx \\
  &= \int_\Omega y \,(D\lambda\circ\Phi)\psi\det(\nabla\Phi) 
   - \det(\nabla\Phi) \nabla\Phi^{-T} \nabla y \,\lambda\circ\Phi \cdot\psi
   - y \Div\big(\det(\nabla\Phi) \nabla\Phi^{-T}\big) \,\lambda\circ\Phi \cdot\psi \\
  &\qquad - y \det(\nabla\Phi) \nabla\Phi^{-T} \nabla\Phi^T(\nabla\lambda\circ\Phi) \cdot\psi \dx
   + \int_{\partial\Omega} y \,(\lambda\circ\Phi)\det(\nabla\Phi)\nabla\Phi^{-T} \normvec \cdot \psi \drx
  \intertext{by using integration by parts. Applying lemma \ref{lem:divDetGradPhi0} we obtain}
  &= \int_{\partial\Omega} y \,(\lambda\circ\Phi)\det(\nabla\Phi)\nabla\Phi^{-T} \normvec \cdot \psi \drx
   - \int_\Omega \det(\nabla\Phi) \nabla\Phi^{-T} \nabla y \,\lambda\circ\Phi \cdot\psi \dx
 \end{align*}
 by using $(D\lambda\circ\Phi)\psi = (\nabla\lambda\circ\Phi)\cdot\psi$
 and $\Div(aB)=a\Div(B) + B(\nabla a)$ for a scalar-valued function $a$ and a matrix-valued function $B$.
 
\end{proof}

\begin{Theorem}
 \label{thm:pushfw}
 Let $\phi:\Omega\ts{n}\to\mathbb R$ and $\lambda:\Omega\ts{n+1}\to\mathbb R$ be sufficiently smooth. Moreover, $F:y\mapsto F(y)$ denotes a vector field depending on $y\ts{n}$, $\transUn{y}$ is given by \eqref{equ:transUn} and $\Omega\ts{n+1}=\transUn{y}(\Omega\ts{n})$.
 Then 
 \begin{align*}
  D_{y\ts{n}} \Big(\int_{\Omega\ts{n+1}} (\phi\circ\inverseTransUn{y}) \lambda\dx \Big) [\psi]
   = & \int_{\partial\Omega\ts{n}} \tau\phi (\lambda\circ\transUn{y}) DF(y\ts{n}) \psi \cdot \normvec \dx \\
  &-\int_{\Omega\ts{n}} \tau(\lambda\circ\transUn{y}) \nabla\phi \cdot  DF(y\ts{n}) \psi \drx
  + \mathcal O(\tau^2)
 \end{align*}
 holds for $\psi:\Omega\ts{n}\to\mathbb R$. 
\end{Theorem}
\begin{proof}
 Due to definition we define
 \begin{align*}
  \int_{\Omega\ts{n+1}} \big(\phi\circ\inverseTransUn{y}\big) \lambda\dx 
  = \int_{\Omega\ts{n}} \phi \big(\lambda\circ\transUn{y}\big) \det(\nabla\transUn{y}) \dx
  =:  \mathcal K(\transUn{y})
 \end{align*}
 The chain rule yields the variation with respect to $y\ts{n}$ as
 \begin{align*}
  D_{y\ts{n}} \mathcal K(\transUn{y}) [\psi_y]
   = D\mathcal K(\transUn{y}) (\partial_{y\ts{n}} \transUn{y})\,[\psi_y].
 \end{align*}
 Using the definition of $\transUn{y}$ yields
 \begin{align*}
  \partial_{y\ts{n}} \transUn{y}\,[\psi_y] = \tau DF(y\ts{n}) [\psi_y]
 \end{align*}
 and applying lemma \ref{lem:dPhiIntegralyLambda} gives
 \begin{align*}
  \partial_{y\ts{n}} \mathcal K [\psi_y]
  = &\int_{\partial\Omega\ts{n}} \phi (\lambda\circ\transUn{y} ) \det(\nabla\transUn{y})\nabla\transUn{y}^{-T}\normvec\cdot \tau DF(y\ts{n})\psi_y \dx\\
  &-\int_{\Omega\ts{n}} \det(\nabla\Phi)\nabla\Phi^{-T} \nabla \phi (\lambda\circ\transUn{y})\cdot \tau DF(y\ts{n}) \psi_y \drx
 \end{align*}
 Since the determinant and inverse of a matrix satisfies
 \begin{align*}
  \det( I + \eps A) = 1 + \eps \trace(A) + \mathcal O(\eps^2)
  \qquad\mbox{and}\qquad
  (I+\eps A)^{-1} = I - \eps A + \mathcal O(\eps^2),
 \end{align*}
 respectively, for small $\eps>0$, we obtain
 \begin{align*}
  \det( \nabla\transUn{y}) = \det( I + \tau F(y\ts{n}) ) = 1 + \tau \Div\big( F(y\ts{n}) \big) + \mathcal O(\tau^2)
 \end{align*}
 and
 \begin{align*}
  (\nabla\transUn{y})^{-1} = (I + \tau \nabla F(y\ts{n}))^{-1} =   I - \tau \nabla F(y\ts{n}) + \mathcal O(\tau^2)
 \end{align*}
 for small time steps $\tau>0$. Therefore, we get
 \begin{align*}
  &\partial_{y\ts{n}} \mathcal K [\psi_y]
  = \int_{\partial\Omega\ts{n}} \phi (\lambda\circ\transUn{y}) \big(1+\tau\Div( F(y\ts{n}))\big)\big(I-\tau \nabla F(y\ts{n})\big)^T\normvec \cdot\tau DF(y\ts{n}) \psi_y\dx \\
  &\quad-\int_{\Omega\ts{n}} \big(1+\tau\Div(F(y\ts{n}))\big)\big(I-\tau \nabla F(y\ts{n})\big)^T \nabla\phi (\lambda\circ\transUn{y})\cdot \tau DF(y\ts{n}) \psi_y \drx
  + \mathcal O(\tau^2)
 \end{align*}
 which finally yields
 \begin{align*}
  \partial_{y\ts{n}} \mathcal K [\psi_y]
  =& \int_{\partial\Omega\ts{n}} \tau\phi (\lambda\circ\transUn{y}) DF(y\ts{n}) \psi_y \cdot \normvec \dx \\
  &-\int_{\Omega\ts{n}} \tau(\lambda\circ\transUn{y}) \nabla\phi \cdot  DF(y\ts{n}) \psi_y \drx
  + \mathcal O(\tau^2)
 \end{align*} 

\end{proof}

\begin{Remark}
 The above theorem holds for arbitrary flux functions $F$, i.e.\ $F:V\ts{n}\to W$ for an appropriate Banach space $W$.
\end{Remark}

The following corollary shows the application of the general theorems stated above to problems given by 
\begin{Corollary}
 \label{cor:pushfw}
 Let $y\in V$, $\phi\in V\ts{n}$ and $\lambda\in V\ts{n}$. Moreover, let $F:\mathbb R\to\mathbb R^d$ and $\transUn{y}$ be given by \eqref{equ:transUn}.
 Then
 \begin{align*}
  \< \partial_\phi (\phi\circ\inverseTransUn{y}) [\psi_\phi], \lambda\>_{H\ts{n+1}}
   = &\< \lambda\circ\transUn{y} , \psi_\phi\>_{H\ts{n}} \\
    &+ \tau \< \lambda\circ\transUn{y} F'(y\ts{n})\cdot \nabla y\ts{n}, \psi_\phi\>_{H\ts{n}} + \mathcal O(\tau^2)
 \end{align*}
 holds for $\psi_\phi\in V\ts{n}$ and 
 \begin{align*}
  \< \partial_{y\ts{n}} (\phi\circ\inverseTransUn{y}) [\psi_y], \lambda\>_{H\ts{n+1}}
   = &\tau \<\phi(\lambda\circ\transUn{y}) F'(y\ts{n})\cdot\normvec, \psi_y \>_{H_\Gamma\ts{n}}\\
   &- \tau \<(\lambda\circ\transUn{y}) \nabla\phi\cdot F'(y\ts{n}), \psi_y \>_{H\ts{n}}  + \mathcal O(\tau^2)
 \end{align*}
 for $\psi_y\in V\ts{n}$.
\end{Corollary}
\begin{proof}
 The first part is given by
 \begin{align*}
  \< \partial_\phi &(\phi\circ\inverseTransUn{y}) [\psi_\phi], \lambda\>_{H\ts{n+1}}
   = \int_{\Omega\ts{n+1}} (\psi_\phi\circ\inverseTransUn{y}) \lambda \dx \\
  &= \int_{\Omega\ts{n}} \psi_\phi (\lambda\circ\transUn{y}) \det(\nabla\transUn{y}) \dx
   = \int_{\Omega\ts{n}} \psi_\phi (\lambda\circ\transUn{y}) \big( 1 + \tau \Div(F(y\ts{n})) \big) \dx + \mathcal O(\tau^2)
 \end{align*}
 which yields the assumption with $\Div(F(y\ts{n})) = F'(y\ts{n})\cdot\nabla y\ts{n}$.
 
 The second part is a direct consequence of theorem \ref{thm:pushfw} by using the fact that $F:\mathbb R\to \mathbb R^d$.
 
\end{proof}

\subsection{Adjoint System}
We apply the above theorems to the optimal control problem \minCont. The derivative with respect to the control function $u$ is given by
\begin{align*}
 \<\partial_{u} &L, \psi_u\>_{U^*,U} = \<\partial_{u} J(y,u), \psi_u\>_{U^*,U}
  - \sum_{n=1}^{N_t} \tau  \<(\mathcal B\ts{n})^* \lambda\ts{n}, \psi_u\>_{U^*,U}
\end{align*}
for all $\psi_u\in U$. The first variation of the Lagrange functional with respect to $y\ts{n}$ is
\begin{align*}
 \<&\partial_{y\ts{n}} L, \psi_y\>_{(V\ts{n})^*,V\ts{n}} = 
 \< \partial_{y\ts{n}} J(y,u), \psi_y \>_{(V\ts{n})^*,V\ts{n}}
  + \<\lambda\ts{n},\psi_y\>_{H\ts{n}}  
  + \tau\< DA(y\ts{n})\psi_y, \lambda\ts{n}\>_{(V\ts{n})^*,V\ts{n}} \\
  &\quad- \Big[ 
     \<\lambda\ts{n+1}\circ\transUn{y}, \psi_y \>_{H\ts{n}} 
     + \tau \< F'(y\ts{n})\cdot \nabla y\ts{n} \psi_y + F(y\ts{n})\cdot\nabla \psi_y, \lambda\ts{n+1}\circ\transUn{y}\>_{H\ts{n}}  \\
     &\quad+ \tau \< F'(y\ts{n}\cdot\nabla y\ts{n} (\lambda\ts{n+1}\circ\transUn{y}), \psi_y\>_{H\ts{n}} 
  \Big]
  - \Big\{
     \tau \<y\ts{n}(\lambda\ts{n+1}\circ\transUn{y}) F'(y\ts{n})\cdot\normvec, \psi_y\>_{H_\Gamma\ts{n}} \\
     &\quad- \tau\<F'(y\ts{n})\cdot\nabla y\ts{n} (\lambda\ts{n+1}\circ\transUn{y}), \psi_y \>_{H\ts{n}} 
  \Big\}
  + \mathcal O(\tau^2)
\end{align*}
for all $\psi_y\in V\ts{n}$ in the corresponding time step $n=1\ldots N_t-1$ by using corollary \ref{cor:pushfw}.
Only the inner products in the $\{\}$ brackets are a result of the implicit variation with respect to the domain.
Terms, which are handled explicitly in the time-discretisation, e.g.\ $\tau F(y)\cdot \nabla y$, have variations of order $\tau^2$ which are directly included in $\mathcal O(\tau^2)$.
All remaining terms are due to the variation of the partial differential equation as usual.
We simplify the above result for $\partial_{y\ts{n}} L$ as
\begin{align*}
 \<&\partial_{y\ts{n}} L, \psi_y\>_{(V\ts{n})^*,V\ts{n}} = 
 \< \partial_{y\ts{n}} J(y,u), \psi_y \>_{(V\ts{n})^*,V\ts{n}}
  + \<\lambda\ts{n},\psi_y\>_{H\ts{n}}  \\
  &\quad+ \tau\< DA(y\ts{n})\psi_y, \lambda\ts{n}\>_{(V\ts{n})^*,V\ts{n}}
  - \<\lambda\ts{n+1}\circ\transUn{y}, \psi_y\>_{H\ts{n}} \\
  &\quad- \tau \< \Div( \psi_y F(y\ts{n})), \lambda\ts{n+1}\circ\transUn{y}\>_{H\ts{n}}   
  - \tau \<y\ts{n}(\lambda\ts{n+1}\circ\transUn{y}) F'(y\ts{n})\cdot\normvec, \psi_y\>_{H_\Gamma\ts{n}} 
  + \mathcal O(\tau^2)
\end{align*}
Furthermore, we obtain for the final time step $N_t$
\begin{align*}
 \< \partial_{y\ts{N_t}} &L, \psi_y\>_{(V\ts{N_t})^*,V\ts{N_t}} = \<\partial_{y\ts{N_t}} J(y,u), \psi_y\>_{(V\ts{N_t})^*,V\ts{N_t}} 
 + \< \lambda\ts{N_t}, \psi_y \>_{H\ts{N_t}} \\
 &+ \tau \< DA(y\ts{N_t})\psi_y, \lambda\ts{N_t}\>_{(V\ts{N_t})^*,V\ts{N_t}}  
\end{align*}
All terms of order $\tau$ are a consequence of the time-implicit scheme we chose for the discretisation and they disappear in explicit schemes. Since we only consider small time steps $\tau\ll 1$, we neglect them in the following for the last time step.
Using the fact that
\begin{align*}
 \< \Div(\psi_y F(y\ts{n})), \lambda\>_{H\ts{n}}
 = \< \lambda F(y\ts{n})\cdot\normvec, \psi_y\>_{H_\Gamma\ts{n}}
 - \< \nabla\lambda \cdot F(y\ts{n}), \psi_y\>_{H\ts{n}}
\end{align*}
hold, we can, roughly, identify the adjoint equation as
\begin{align*}
 \frac1\tau( \lambda\ts{n}-\lambda\ts{n+1}\circ\transUn{y}) + DA(y\ts{n})^* \lambda\ts{n} 
 +\frac1\tau \partial_{y\ts{n}} J(y,u) = \nabla&(\lambda\ts{n+1}\circ\transUn{y})\cdot F(y\ts{n}) \\
 &+ \mathcal C\ts{n}(\lambda\ts{n+1}\circ\transUn{y})
\end{align*}
in $(V\ts{n})^*$ and hence an equation in $\Omega\ts{n}$. Here $\mathcal C\ts{n}:V\ts{n}\to (V\ts{n})^*$ is given by
\begin{align*}
 \<\mathcal C\ts{n} (\lambda\ts{n+1}\circ\transUn{y}), \psi\>_{(V\ts{n})^*,V\ts{n}} := \<(\lambda\ts{n+1}\circ\transUn{y}) \big(y\ts{n} F'(y\ts{n}) + F(y\ts{n})\big)\cdot\normvec , \psi\>_{H\ts{n}}
\end{align*}
and represents the boundary values. It is, among others, a consequence of the implicit domain variation.
The gradient of the reduced cost functional $\hat J(u) := J(y(u),u)$ is given by
\begin{align*}
 \nabla \hat J(u) := \rieszIso^{-1}_U\left(
  \frac1\tau \partial_{u} J(y,u) - \sum_{n=1}^{N_t} (\mathcal B\ts{n})^* \lambda\ts{n}
 \right)
\end{align*}
where $\rieszIso_U$ denotes the Riesz isomorphism $\rieszIso_U:U\to U^*$.

\subsection{Conclusion}
The optimisation approach described in this section is very close to the numerical implementation of free boundary problems.
The basic principle is illustrated in figure~\ref{fig:transformationOptimisation}. 
Here, the state (or forward) equation yields the domains $\Omega\ts{n}$ recursively by the transformation $\transUn{y}$. Consequently, the adjoint (or backward) equation uses the same domains, starting from $\Omega\ts{N_t}$ to $\Omega\ts{0}$.
Particularly, the push forward term of the discrete time derivative in the forward problem changes to a pullback term in the adjoint equation.

\begin{figure}[bt]
 \centering
 \psfrag{A}{$\Omega\ts{0}$}
 \psfrag{B}{$\Omega\ts{1}$}
 \psfrag{C}{$\Omega\ts{2}$}
 \psfrag{P1}{$\Phi\ts{0}$}
 \psfrag{P2}{$\Phi\ts{1}$}
 \psfrag{P3}{$\Phi\ts{2}$}
 \includegraphics[width=.82\textwidth]{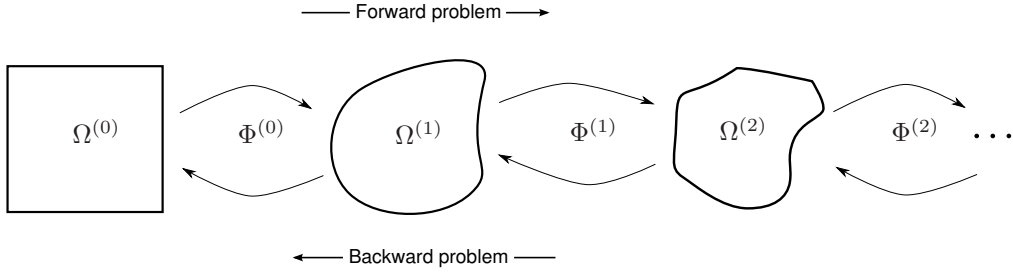}
 \caption{Basic principle of the time-discrete optimisation of free boundary problems. The state equation yields the sequence of domains $\Omega\ts{n}$ as a consequence of $\transUn{y}$. The adjoint system goes backwards, starting from $\Omega\ts{N_t}$ to $\Omega\ts{0}$.}
 \label{fig:transformationOptimisation}
\end{figure}

Instead of deriving the variation of each domain $\Omega\ts{n}$ directly, it is handled by the variation of the push forward terms, e.g.\ in the time-derivative, with respect to $\transUn{y}$ and hence to $y$.
This procedure allows a easy derivation of the adjoint equation as only a few simple (explicit) terms are needed to obtain information about the shape dependency.
Moreover, we showed above that terms, handled by an explicit time-discretisation, do not yield a contribution to the domain variation as the derivatives are mainly of order $\tau^2$. The next section shows some numerical examples applying the above method.



\section{Numerical Examples}
\label{sec:numRes}
In this section we verify the optimisation approach shown in the previous sections numerically.
The first test case involve the Navier-Stokes equation with a free surface, c.f.\ \cite{Beale3}, where the transformation is given in a natural way, that is, the flow velocity is used. The second example is a Stefan-type problem where the final shape of a melting or solidification process is optimised.
Here, the transformation is not given by a flow velocity but an artificial motion, similar to an ALE \cite{ALE2} method.

\subsection{Navier-Stokes equation}
In this example, we optimise the filling of an open liquid tank with small obstacles shown in figure~\ref{num:fig:domLiqTank}. 
In particular, we search an inflow profile such that for a given outflow profile the free surface remains as still as possible.
This yields the optimisation problem: Minimise
\begin{align*}
 J(\vecb u, \Omega) := \frac12 \int_0^T \int_{\Gamma_f^t} |\vecb u(x,t)|^2 \drx \,dt
\end{align*}
subject to the Navier-Stokes equation
\begin{subequations}
\label{equ:ns_cont}
\begin{align}
 \partial_t \vecb u + \vecb u \cdot \nabla\vecb u - \nu\Delta\vecb u &= -\nabla p
 &&\mbox{in }\Omega_t\times(0,T) \\
 \Div \vecb u &= 0
 &&\mbox{in }\Omega_t\times(0,T) \\ 
 \vecb u &= 0
 &&\mbox{on }\Gamma_w \times(0,T) \\
 \vecb u &= \vecb u_o
 &&\mbox{on }\Gamma_o \times(0,T) \\
 \nabla p\cdot\normvec &= c
 &&\mbox{on }\Gamma_i \times(0,T) \label{num:equ:nsInflowVelo}\\
 \nabla \vecb u\cdot \normvec &= 0
 &&\mbox{on }\Gamma_f^t \times(0,T) \label{num:equ:nsFreeSurfCoupU}\\ 
 p &= 0
 &&\mbox{on }\Gamma_f^t \times(0,T) \label{num:equ:nsFreeSurfCoupP}\\ 
 \vecb u(0) &= 0
 &&\mbox{in }\Omega_0.
\end{align}
\end{subequations}

The domain $\Omega_t$ depends on the velocity profile $\vecb u$ at the free surface $\Gamma_f^t$. In particular, the motion of $\Gamma_f^t$ is given by $\vecb u|_{\Gamma_f^t}$.
For convenience we use \eqref{num:equ:nsInflowVelo} as inflow condition. Here, $c$ denotes the control variable depending on space and time, i.e.\ $c:\Gamma_i \times[0,T] \to \mathbb R$. From a physical point of view, $\nabla p\cdot \normvec$ is proportional to the inflow flux in normal direction.
The condition for the free surface, \eqref{num:equ:nsFreeSurfCoupU} and \eqref{num:equ:nsFreeSurfCoupP}, is described in more detail in e.g.\ \cite{Ferziger08}.

\begin{figure}[tb]
 \centering
 \begin{minipage}{.49\textwidth}
  \centering
  \psfrag{A}[c]{$\Gamma_i$}
  \psfrag{B}{$\Gamma_w$}
  \psfrag{C}[c]{$\Gamma_w$}
  \psfrag{D}[c]{$\Gamma_w$}
  \psfrag{E}{$\Gamma_f$}
  \psfrag{F}{$\Gamma_o$}
  \includegraphics[scale=0.5]{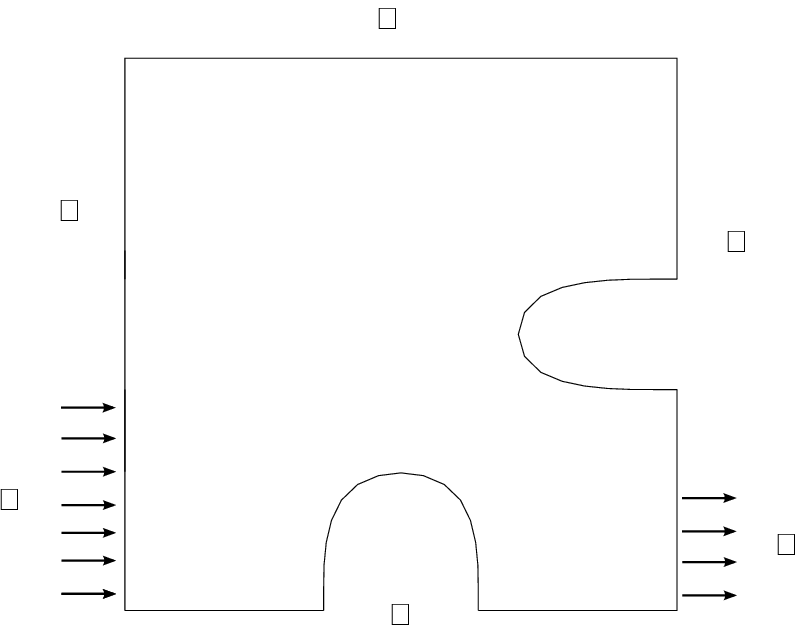}
  \caption{Domain for the Navier-Stokes example.}
  \label{num:fig:domLiqTank}
 \end{minipage}
 \begin{minipage}{.49\textwidth}
  \centering
  \includegraphics[scale=1.0]{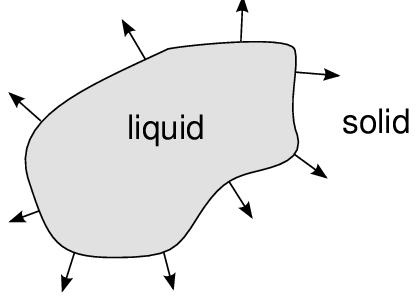}
  \caption{Illustration of a Stefan-problem domain.}
  \label{fig:stefanDomain}
 \end{minipage}
\end{figure}

\subsubsection{Time Discretisation}
For the time-discretisation a Chorin projection \cite{Chorin, kuhnertChorin} is used.
In particular, we choose the domain transformation $\Phi\ts{n}:\Omega\ts{n}\to\Omega\ts{n+1}$ as
\begin{align*}
 \transUn{\vecb u} &:= \Id{}_{\Omega\ts{n}} + \tau \vecb u\ts{n}.
\end{align*}
Then, the time discretisation of \eqref{equ:ns_cont} reads
\begin{align*}
 \frac{1}{\tau}\big(\vecb u\ts{n+1} - (\vecb u\ts{n}\circ\inverseTransUn{\vecb u})\big) - \nu\Delta\vecb u\ts{n+1} &= -\nabla p\ts{n+1} 
 &&\mbox{in }\Omega\ts{n+1}\\
 \Delta p\ts{n+1} &= \frac{1}{\sigma}\Div \vecb u\ts{n+1}  
 &&\mbox{in }\Omega\ts{n+1} \\
 \vecb u\ts{n+1} &= 0 
 &&\mbox{on }\Gamma_w \\
 \vecb u\ts{n+1} &= \vecb u_o
 &&\mbox{on }\Gamma_o \\ 
 p\ts{n+1} &= 0
 &&\mbox{on }\Gamma_f\ts{n+1} \\ 
 \nabla p\ts{n+1}\cdot \normvec &= c\ts{n+1}
 &&\mbox{on }\Gamma_i\\
 \vecb u\ts{0} &= 0
 &&\mbox{in }\Omega\ts{0}.
\end{align*}
where $\sigma>0$ denotes a regularisation parameter.
All boundaries not mentioned above are set to uniform Neumann condition, i.e.\ $\nabla p\cdot\normvec = 0$ or 
$\nabla\vecb u\cdot\normvec = 0$, which is a consequence of Chorin's projection.
The cost function is discretised straight forward
\begin{align*}
 J(\vecb u, c) := \frac12\sum_{n=1}^{N_t-1} \tau \int_{\Gamma_f\ts{n}} |\vecb u\ts{n}|^2 \drx .
\end{align*}

\subsubsection{Adjoint system}
Using the previous optimisation approach, we obtain the identification of the adjoint system as

\begin{align*}
 \begin{aligned}
 \frac1\tau  \big( \lambda_{\vecb u}\ts{n} - (\lambda_{\vecb u}\ts{n+1}\circ\transUn{\vecb u}) \big)
 - \nu \Delta\lambda_{\vecb u}\ts{n} &= \nabla \lambda_p\ts{n} 
 - \big(\lambda_{\vecb u}\ts{n+1}\circ\transUn{\vecb u}\big)\cdot \nabla \vecb u\ts{n}
 &&\mbox{in }\Omega\ts{n} \\
 -\sigma \Delta \lambda_p\ts{n} &= \Div \lambda_{\vecb u}\ts{n}
 &&\mbox{in }\Omega\ts{n} \\
 \nu\nabla\lambda_{\vecb u}\ts{n} \,\normvec &= - \vecb u\ts{n}
  - \big(\vecb u\ts{n}\cdot(\lambda_{\vecb u}\ts{n+1}\circ\transUn{\vecb u})\big)\normvec
 &&\mbox{on }\Gamma_f\ts{n} \\
 \lambda_p\ts{n} &= 0
 &&\mbox{on }\Gamma_f\ts{n}\\
 \nu\nabla\lambda_{\vecb u}\ts{n} \,\normvec &= -\lambda_p\ts{n} \normvec
 - \big(\vecb u\ts{n}\cdot(\lambda_{\vecb u}\ts{n+1}\circ\transUn{\vecb u})\big)\normvec 
 &&\mbox{on }\Gamma_i \\
 \sigma\nabla \lambda_p \cdot\normvec &= -\lambda_{\vecb u}\ts{n}\cdot\normvec
 &&\mbox{on }\Gamma_i \\
 \lambda_{\vecb u}\ts{n} &= 0
 &&\mbox{on }\Gamma_{w,o} \\
 \sigma\nabla \lambda_p \cdot\normvec &= 0
 &&\mbox{on }\Gamma_{w,o} \\
 \lambda_{\vecb u}\ts{N_t} &= 0
 &&\mbox{in }\Omega\ts{N_t}.
 \end{aligned}
\end{align*}
The gradient of the reduced cost functional $\hat J(c) := J(\vecb u(c),c)$ is identified by
\begin{align*}
 \nabla \hat J(c) := - \sigma \lambda_p|_{\Gamma_i}.
\end{align*}
A detailed derivation of the adjoint system can be found in \cite{MyPhd}.
Note that only the pullback terms on the right hand side, e.g.\ $(\lambda_{\vecb u}\ts{n+1}\circ\transUn{\vecb u})\cdot \nabla \vecb u\ts{n}$, are a result of the domain variation.

\subsubsection{Numerical Results}
For the numerical implementation we use a meshless particle method, cf.\ \cite{MyPhd}. These methods are superior to mesh-based methods like finite elements as no connectivity information is used and hence no expensive remeshing is needed if the positions of the supporting points change. Particularly, we use the finite pointset method, cf.\ \cite{kuhnertGridfree}. The optimisation process is performed by a second order BFGS method with Armijo rule, cf.\ \cite{Deuflhard}.

The final time is set to $T=2.5$, the step size $\tau=0.005$. The domain has a width and height of $5.0$. Moreover, the viscosity is set to $\nu=10$ and $\sigma=0.005$. The outflow velocity is given by a parabolic profile with $u_{max}=3.0$ in order to avoid singularities at the corners.

Figure~\ref{num:fig:fillingUncont} and \ref{num:fig:fillingOptim} show the results for the uncontrolled and optimised case, respectively. Here the colour represents the velocity magnitude. The uncontrolled case, that is, no inflow is given, yields a high velocity at the free boundary and therefore a large deformation of it, as expected. The resulting behaviour of the free surface is an effect of the high viscosity of the fluid. 
The controlled case forms out a straight flow, as good as possible for the given geometry, from the inflow to the outflow, which does not affect the free surface strongly. Hence, the surface velocity is much smaller than for the uncontrolled case, cf.\ figure~\ref{fig:ns_velosurf}. Note that due to the setting, the desired value of the velocity $\vecb u = 0$ at the free surface is not reachable and hence $\min J=0$ cannot be expected.
The corresponding evaluation of the cost functional and gradient norm is shown in figure~\ref{fig:ns_convergence}. The gradient norm shows a strong decrease in the first iterations, then it becomes flatter with order $1.04$ to $1.25$. Note that this is not surprisingly as the problem is highly non convex. On the one hand, the Navier-Stokes equation yields a complicated constraint, on the other hand, the adaptation can cause problems for the optimisation process.

\subsection{Stefan-type Problem}
The second example deals with a Stefan-type problem \cite{Visintin2008377}. These problems describe phase transiations, e.g\ from liquid to solid, see figure~\ref{fig:stefanDomain}. Particularly, we solve the heat equation
\begin{align}
 \label{equ:stefanProblAllg}
 \begin{aligned}
  \partial_t \theta - \Delta\theta &= f
  &&\mbox{in }\Omega_t \times(0,T) \\
  \theta &= \theta_m
  &&\mbox{on }\Gamma_t\times(0,T) \\
  \theta(0) &= \theta_0
  &&\mbox{in }\Omega_0\\
  V_n &= \beta \nabla\theta\cdot\normvec 
 \end{aligned}
\end{align}
with $\theta_m\in\mathbb R$ as the boundary temperature (melting point) and $\theta_0$ as the initial temperature distribution. $V_n$ denotes the velocity of the boundary in normal direction and hence also defines $\Omega_t$. Moreover, $\beta\in\mathbb R$ denotes a material constant.
Note that the sign of $\beta$ indicates whether $\Omega_t$ is a liquid or solid phase.

For the minimisation we start with an arbitrary domain $\Omega_0$ with boundary $\Gamma_0$ and want to achieve a desired shape with boundary $\Gamma_d$ at final time $T$.
In the following, we choose an ellipse for $\Gamma_d$. This desired boundary shape is parametrised in order to formulate an optimisation problem.
The easiest way of a parametrisation is the implicit definition of an ellipse given by
\begin{align}
 \label{equ:implicitEllipse}
 \frac{x^2}{a^2} + \frac{y^2}{b^2} = 1
\end{align}
in two dimensions. Here, $a$ and $b$ denote the stretching in the $x$- and $y$-direction, respectively.
Hence, the domain $\Omega_T$ has the boundary $\Gamma_d$ if \eqref{equ:implicitEllipse} is satisfied for all points on the boundary $\Gamma_T$.
The minimisation problem can therefore be interpreted as: Find a heat source $f$ such that
\begin{align}
 \label{equ:ellipseCostIntegral}
 \frac12 \int_{\Gamma_T} \left( \frac{x^2}{a^2}+\frac{y^2}{b^2}-1 \right)^2\drx
\end{align}
is minimal subject to \eqref{equ:stefanProblAllg}. The integral is minimal, in particular zero, if all boundary points satisfy \eqref{equ:implicitEllipse}. For all other settings, i.e. a measurable amount of points do not lay on the desired boundary, the integral is greater than zero.

\subsubsection{Time-discretisation}
For convenience, we replace the Dirichlet boundary condition of equation \eqref{equ:stefanProblAllg} by a Robin condition, i.e.
\[ \theta = \theta_m \qquad\rightarrow\qquad \theta + \kappa\nabla\theta\cdot\normvec = \theta_m \]
on $\Gamma_t$ where $\kappa>0$ denotes a small regularisation parameter depending on the discretisation.
Since we are dealing with a problem without convection and we only know the motion of the boundary, the choice of the transformation is not obvious.
In particular, we choose a smooth continuation of the boundary velocity by solving a Laplace equation similar to the ALE method, cf.\ \cite{ALE2}.
\begin{align*}
 \begin{aligned}
  -\Delta\vecb u &= 0
  &&\mbox{in }\Omega_t\\
  \vecb u + \kappa\nabla\vecb u\cdot\normvec &= \beta(\nabla\theta\cdot\normvec)\normvec
  \qquad&&\mbox{on }\Gamma_f
 \end{aligned}
\end{align*}
where we use a regularisation of the Dirichlet boundary condition as before.

We obtain the time discrete system by first solving the velocity equation
\begin{align*}
 \begin{aligned}
  -\Delta\vecb u\ts{n} &= 0
  &&\mbox{in }\Omega\ts{n}\\
  \vecb u\ts{n} + \kappa\nabla\vecb u\ts{n}\cdot\normvec &= \frac{\beta}{\kappa}(\theta_m-\theta\ts{n})\normvec
  \qquad&&\mbox{on }\Gamma\ts{n}
 \end{aligned}
\end{align*}
then the transformation
\begin{align*}
 \Phi\ts{n} &= \Id{}_{\Omega\ts{n}} + \tau \vecb u\ts{n}
 &&\mbox{in }\Omega\ts{n}
\end{align*}
and finally the heat equation
\begin{align*}
 \frac1\tau (\theta\ts{n+1}-(\theta\ts{n}\circ\inverseTrans{n}))  - \Delta\theta\ts{n+1} 
  &= (\vecb  u\ts{n}\cdot\nabla\theta\ts{n})\circ\inverseTrans{n} + \chi\, c\ts{n+1}
 &&\mbox{in }\Omega\ts{n+1} \\
 \theta\ts{n+1} + \kappa\nabla\theta\ts{n+1}\cdot\normvec &= \theta_m
  &&\mbox{on }\Gamma_w\ts{n+1} \\
  \intertext{with the initial condition}
  \theta\ts{0} &= \theta_0
  &&\mbox{in }\Omega\ts{0} .
\end{align*}
where $\chi$ denotes a spatial localisation function depending on $\vecb x$ only.
Note that we use the Robin condition of the heat equation to replace $\nabla\theta\cdot\normvec$ in the velocity equation and add a convection term to the right hand side of the heat equation due to the artificial transformation of the domain.

The discrete cost functional reads
\begin{align*}
 J(\vecb u,c) := \frac12 \int_{\Gamma\ts{N_t}} \Big( \big(\transUN{N_t}{\vecb u}\big)^T \,\mathbf E\, \big(\transUN{N_t}{\vecb u}\big) - 1 \Big)^{2\alpha} \drx
\end{align*}
with
\begin{align*}
 \mathbf E = \begin{pmatrix}
              a^{-2} & 0 \\
              0 & b^{-2}
             \end{pmatrix}
\end{align*}
which corresponds to \eqref{equ:ellipseCostIntegral}. Note that we use the power of $2\alpha$ in the cost functional to be able to consider also $L^p$ norms.

\subsubsection{Adjoint system}
Using the previous optimisation approach, we identify the adjoint equation and gradient as
\begin{align*}
 -\Delta \lambda_{\vecb u}\ts{n} &= \big(\lambda_\theta\ts{n+1}\circ\transUn{\vecb u}\big) \nabla\theta\ts{n}  
 &&\mbox{in }\Omega\ts{n}\\
 \lambda_{\vecb u}\ts{n} + \kappa \nabla\lambda_{\vecb u}\ts{n} \cdot\normvec &= 0
 &&\mbox{on }\Gamma\ts{n}\\[0.5ex]
 \frac1\tau\big( \lambda_\theta\ts{n} - \lambda_\theta\ts{n+1}\circ\transUn{\vecb u}\big) - \Delta \lambda_\theta\ts{n}
  &= -\vecb u\ts{n}\cdot\nabla(\lambda_\theta\ts{n+1}\circ\transUn{\vecb u})
  &&\mbox{in }\Omega\ts{n} \\
 \lambda_\theta\ts{n} + \kappa\nabla\lambda_\theta\ts{n}\cdot\normvec
  &= \kappa(\lambda_\theta\ts{n+1}\circ\transUn{\vecb u})\vecb u\ts{n}\cdot\normvec
  - \frac{\beta}{\kappa} \lambda_{\vecb u}\ts{n} \cdot \normvec
  &&\mbox{on }\Gamma\ts{n}
\end{align*}
for $n=1,\ldots,N_t-1$. For $n=N_t$ we get
\begin{align*}
 \begin{aligned}
 -\Delta \lambda_{\vecb u}\ts{N_t} &= 0
 &&\mbox{in }\Omega\ts{N_t} \\
 \lambda_{\vecb u}\ts{N_t} + \kappa \nabla \lambda_{\vecb u}\ts{N_t}\cdot\normvec &= 
 -2\alpha\kappa\Big( \big(\transUN{N_t}{\vecb u}\big)^T \,\mathbf E\, \big(\transUN{N_t}{\vecb u}\big) - 1 \Big)^{2\alpha-1}
 \,\mathbf E\, (\transUN{N_t}{\vecb u})
 \quad&&\mbox{on }\Gamma\ts{N_t}
 \end{aligned}
\end{align*}
and
\begin{align}
 \label{num:equ:finalTimeAdjTheta}
 \begin{aligned}
 \lambda_\theta\ts{N_t} - \tau \Delta \lambda_\theta\ts{N_t} &= 0 
 &&\mbox{in }\Omega\ts{N_t}\\
 \lambda_\theta\ts{N_t} + \kappa\nabla\lambda_\theta\ts{N_t}\cdot\normvec &= -\frac{\beta}{\kappa} \lambda_{\vecb u}\ts{N_t}\cdot\normvec
 \qquad&&\mbox{on }\Gamma\ts{N_t}.
 \end{aligned}
\end{align}
The gradient of the reduced cost functional reads
\begin{align*}
 \nabla \hat J(c)\ts{n} := - \int_{\Omega\ts{n}} \chi(\vecb x)\, \lambda_\theta\ts{n}(\vecb x) \dx .
\end{align*}
Again, a more detailed derivation of the adjoint equations can be found in \cite{MyPhd}.

\subsubsection{Numerical Results}
The spatial discretisation is, similar to the previous test case, performed by a meshless method with adaptation. Again, the BFGS method with Armijo rule is used for the optimisation. 
The setting is $\alpha=2$, $\beta=-1$, $\kappa=0.01$ and $\theta_m=0$. Moreover, we set the time step size $\tau=0.01$ and the finial time to $T=0.3$.
The localisation function $\chi$ is given by
\begin{align*}
 \chi(\vecb x)\, c\ts{n} := \sum_{i=1}^{N_x} \sum_{j=1}^{N_y} c_{ij}\ts{n} \exp\Big( -25\big( (x-x_{ij})^2 + (y-y_{ij})^2 \big) \Big)
\end{align*}
for the supporting points $-1, -0.6, -0.2 , 0.2, 0.6, 1.0$, i.e.\ $x_{11}=-1, y_{11}=-1$ till $x_{66}=1,y_{66}=1$.
Furthermore, the initial domain is given by a square with edge length $0.6$ and the desired shape is given by a circle of radius $1$ which yields $a=1.0$ and $b=1.0$.
Note that this example does not base on a physical setting as we only want to show the feasibility of our method.

The results are shown in figure~\ref{fig:stefanCircle}, where the colour represents the absolute value	of the source induced by the control function, i.e.\ $\chi\,c$. Moreover, the black line denotes the desired shape at final time. The optimised case shows a small deformation in the first time steps, which becomes larger close to the final time. The shape at final time is very close to the desired one.
The convergence, shown in figure~\ref{fig:stefanCircleConv}, is approximately of order $1.5$ and no Armijo step size reduction is needed.
Furthermore, the plateaus in the the cost functional and gradient norm are due to a reset of the BFGS matrix, implemented for stability reasons.
Note that we neglect these plateaus for the consideration of the convergence order.

\section{Conclusion}
In this paper we presented a simple approach for the optimisation of free boundary problems which is very close the their numerical implementation. In particular, we applied an extended time-discretisation based on an operator splitting. For this a decoupling of the deformation of the domain and the solution of the remaining partial differential equation was performed. 
With this approach, adjoint-based optimisation can be applied easily.
Since we do not perform an explicit derivative with respect to the domain, this information is partly obtained by the variation of the push forward terms needed for the discrete time derivative.
The numerical results based on the Navier-Stokes equation and a Stefan-type problem showed that this optimisation approach yields good results and is very easy to implement numerically.


\begin{figure}[p]
 \centering
 \includegraphics[width=.3\textwidth]{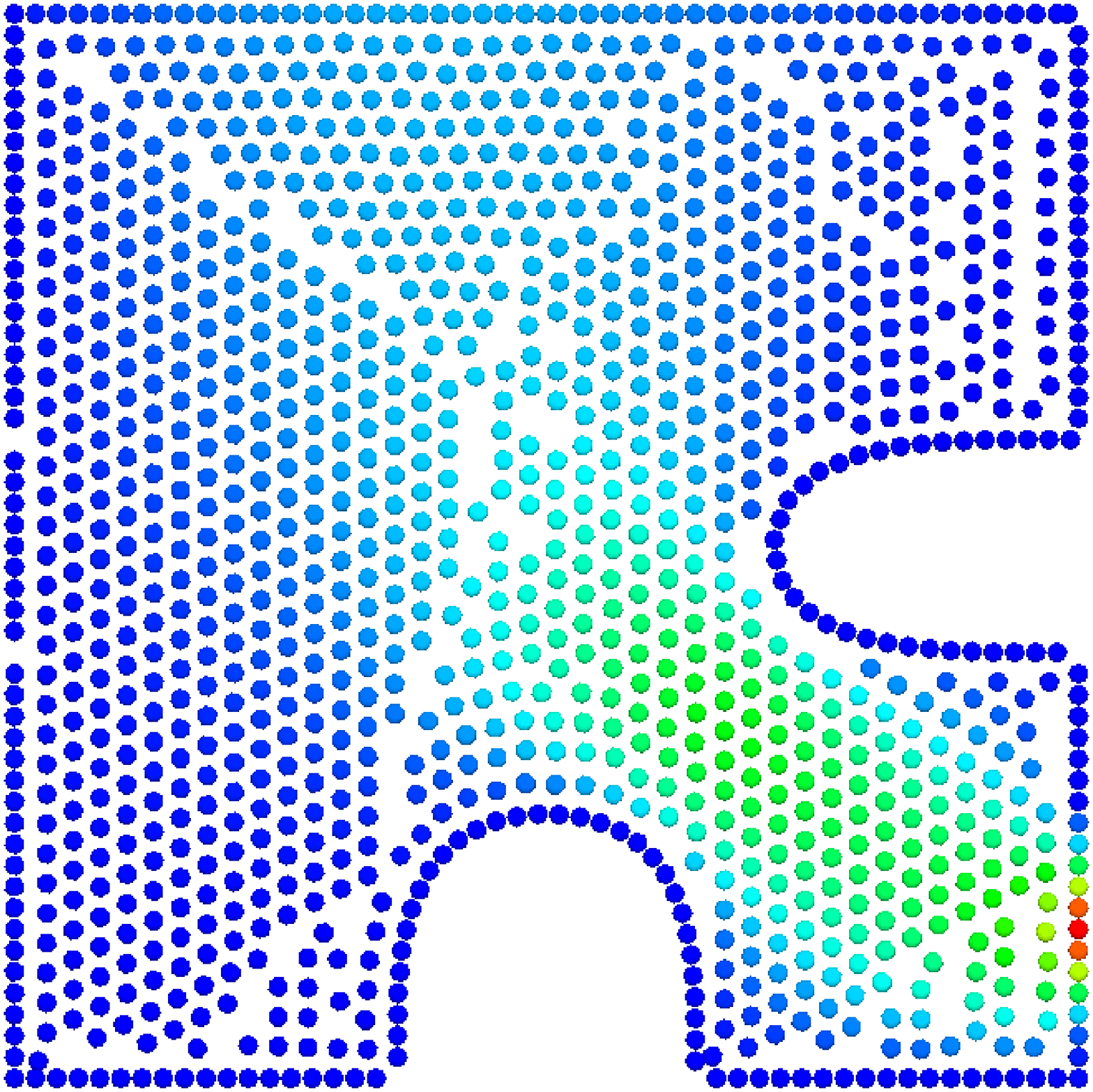}
 \includegraphics[width=.3\textwidth]{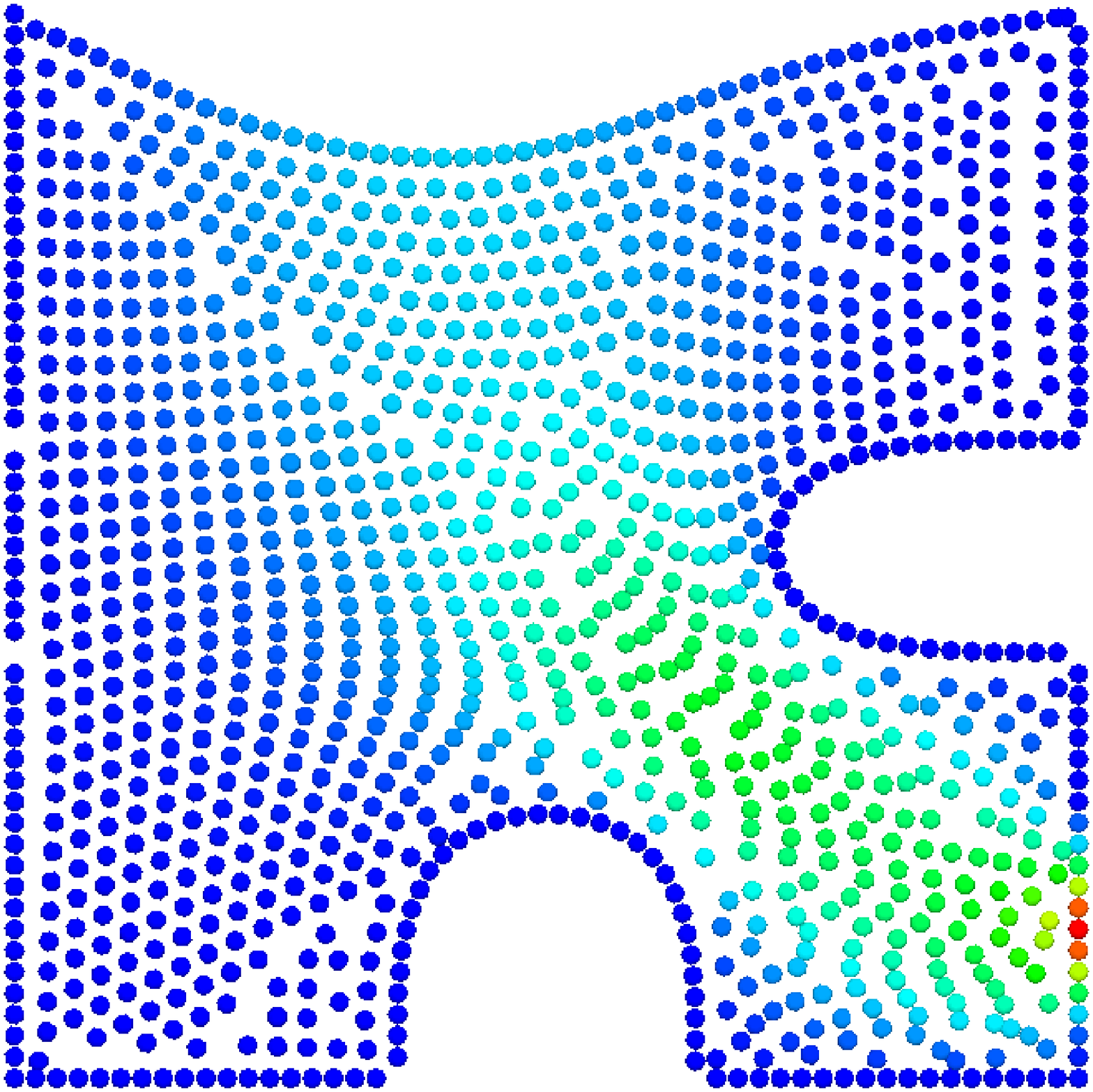}
 \includegraphics[width=.3\textwidth]{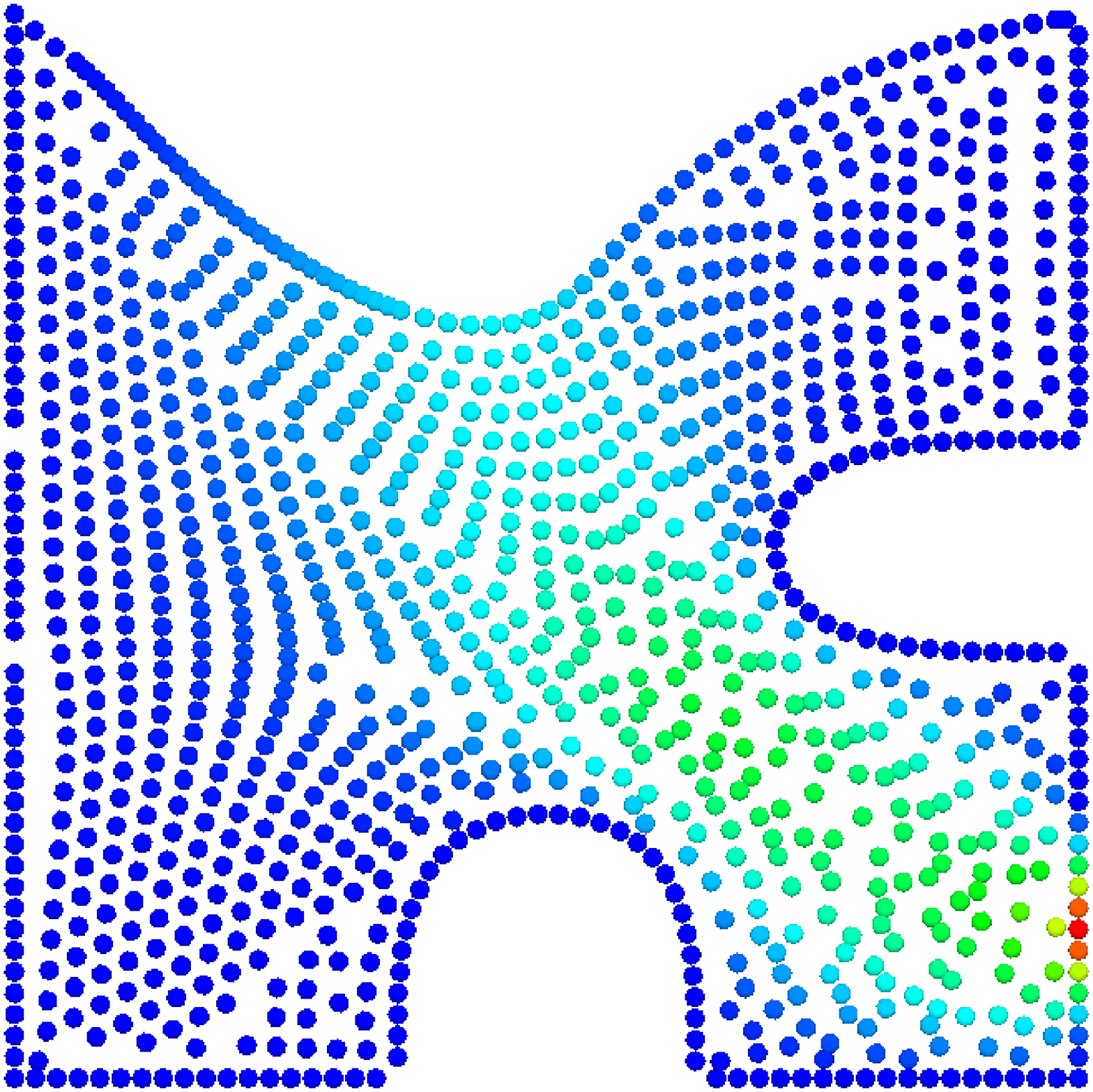}
 \caption{Filling of a liquid tank at time $0, 1.25, 2.5$ for the uncontrolled case.}
 \label{num:fig:fillingUncont}
\end{figure}

\begin{figure}[p]
 \centering
 \includegraphics[width=.3\textwidth]{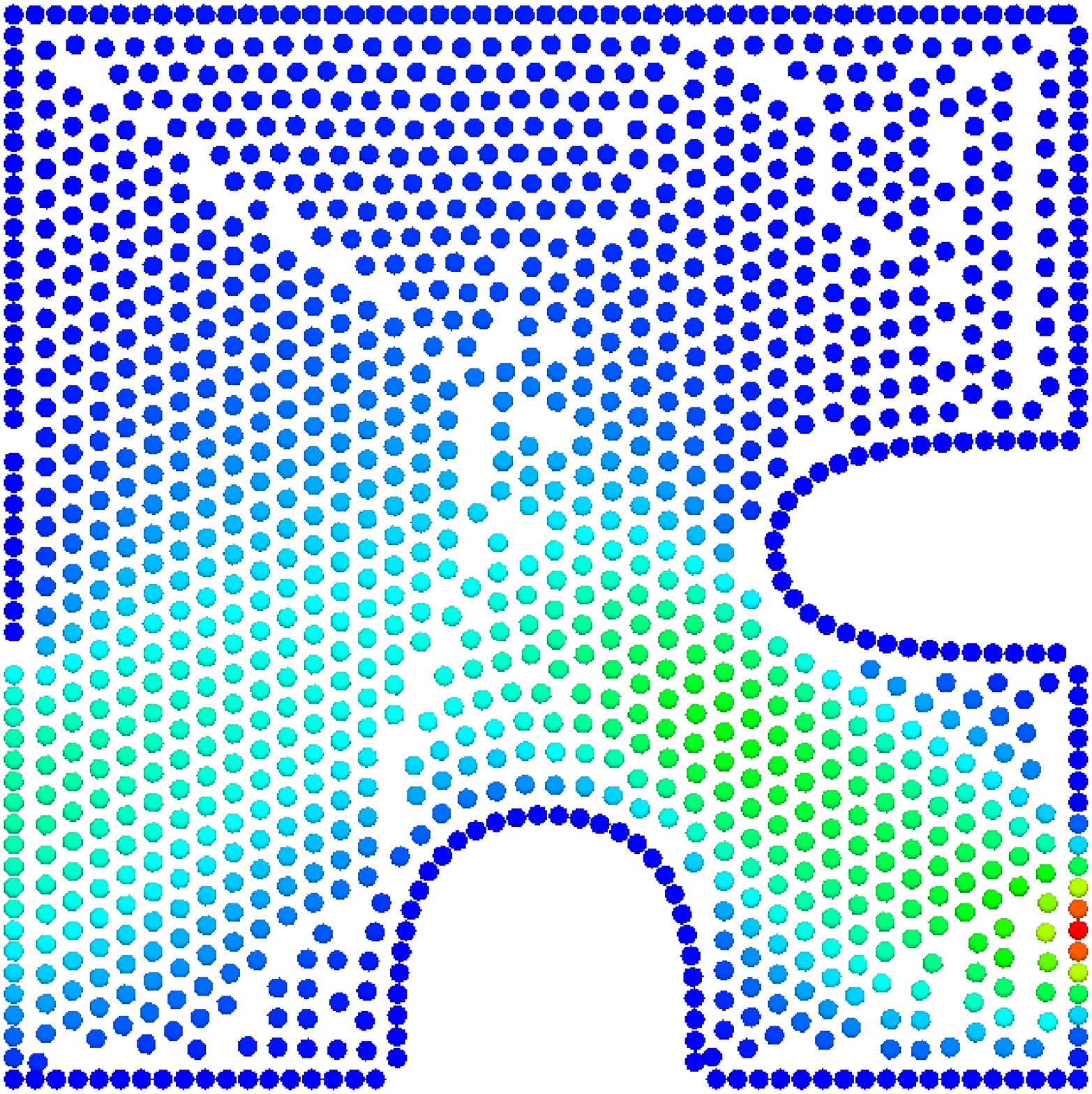}
 \includegraphics[width=.3\textwidth]{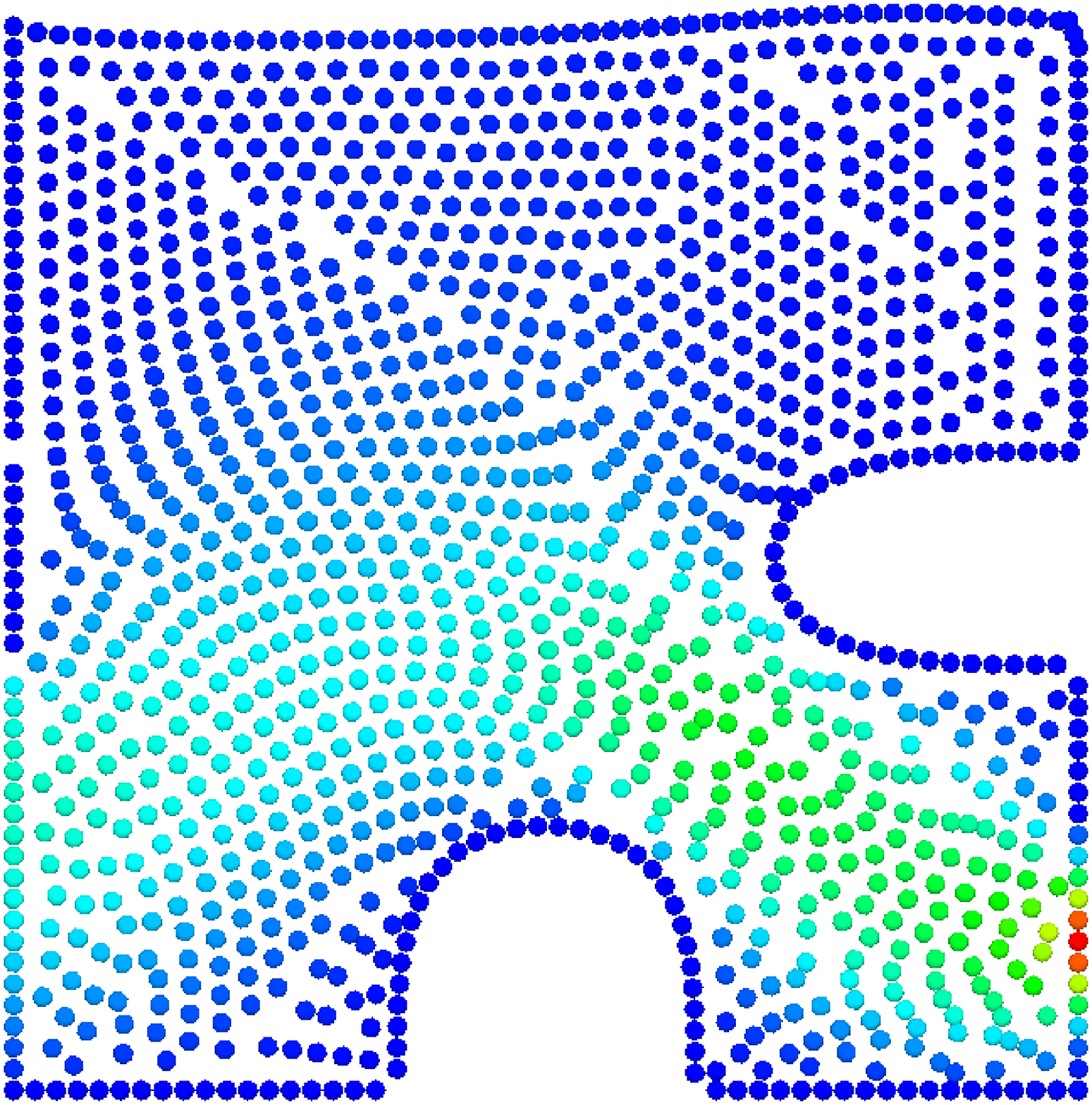}
 \includegraphics[width=.3\textwidth]{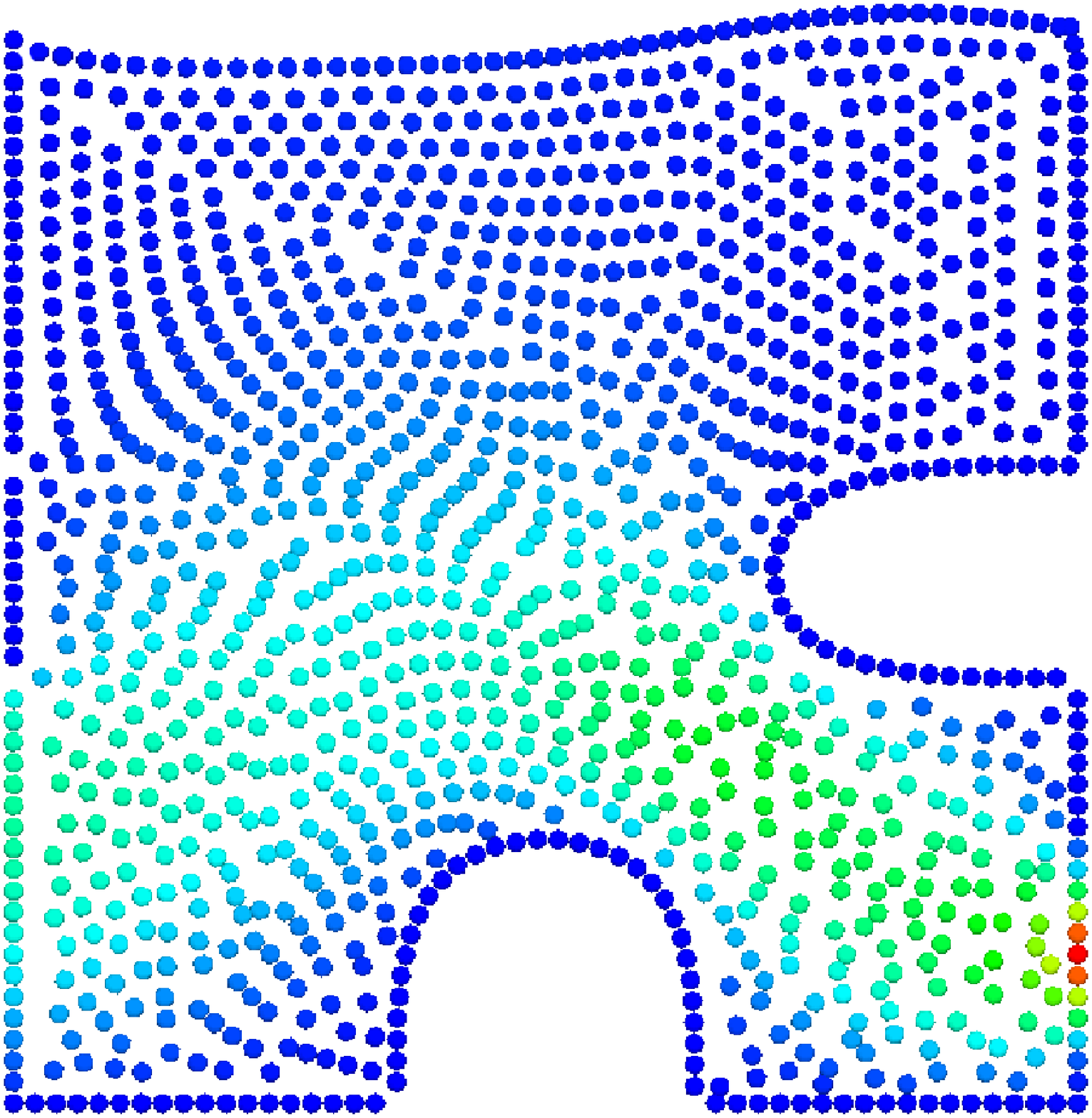}
 \caption{Filling of a liquid tank at time $0, 1.25, 2.5$ for the optimised case.}
 \label{num:fig:fillingOptim}
\end{figure}

\begin{figure}[p]
 \begin{minipage}{.46\textwidth}
  \centering
  \includegraphics[width=\textwidth]{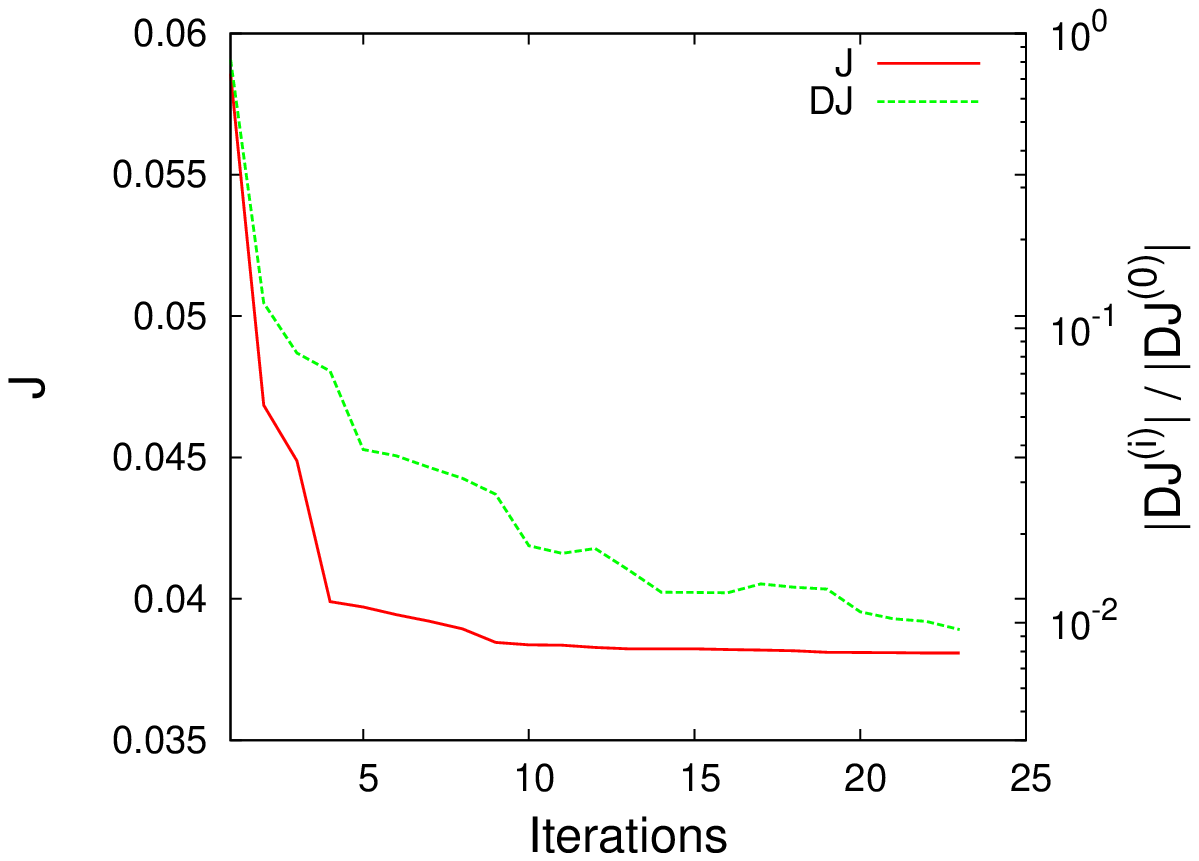}
  \caption{Cost functional and relative gradient norm.}
  \label{fig:ns_convergence}
 \end{minipage}
 \hspace*{8mm}
 \begin{minipage}{.46\textwidth}
  \centering
  \psfrag{Y}{\small $\|\vecb u\|_{L^2(\Gamma_f)}$}
  \includegraphics[width=\textwidth]{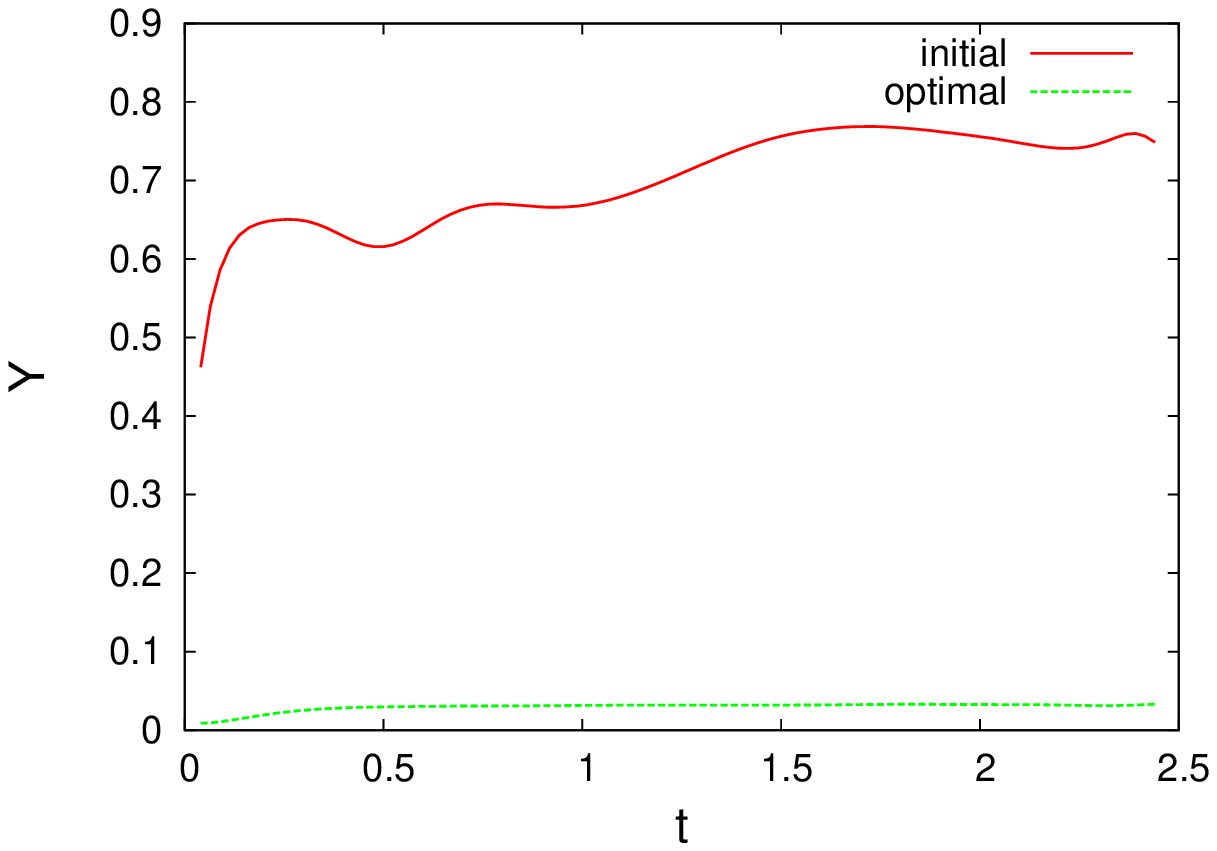}
  \caption[Mean surface velocity $\|\vecb u\|_{L^2(\Gamma_f)}$ over time.]
   {Mean surface velocity $\|\vecb u\|_{L^2(\Gamma_f)}$ over time for the uncontrolled and controlled case.}  
  \label{fig:ns_velosurf}
 \end{minipage}
\end{figure}

\begin{figure}[p]
 \centering
 \includegraphics[width=52mm]{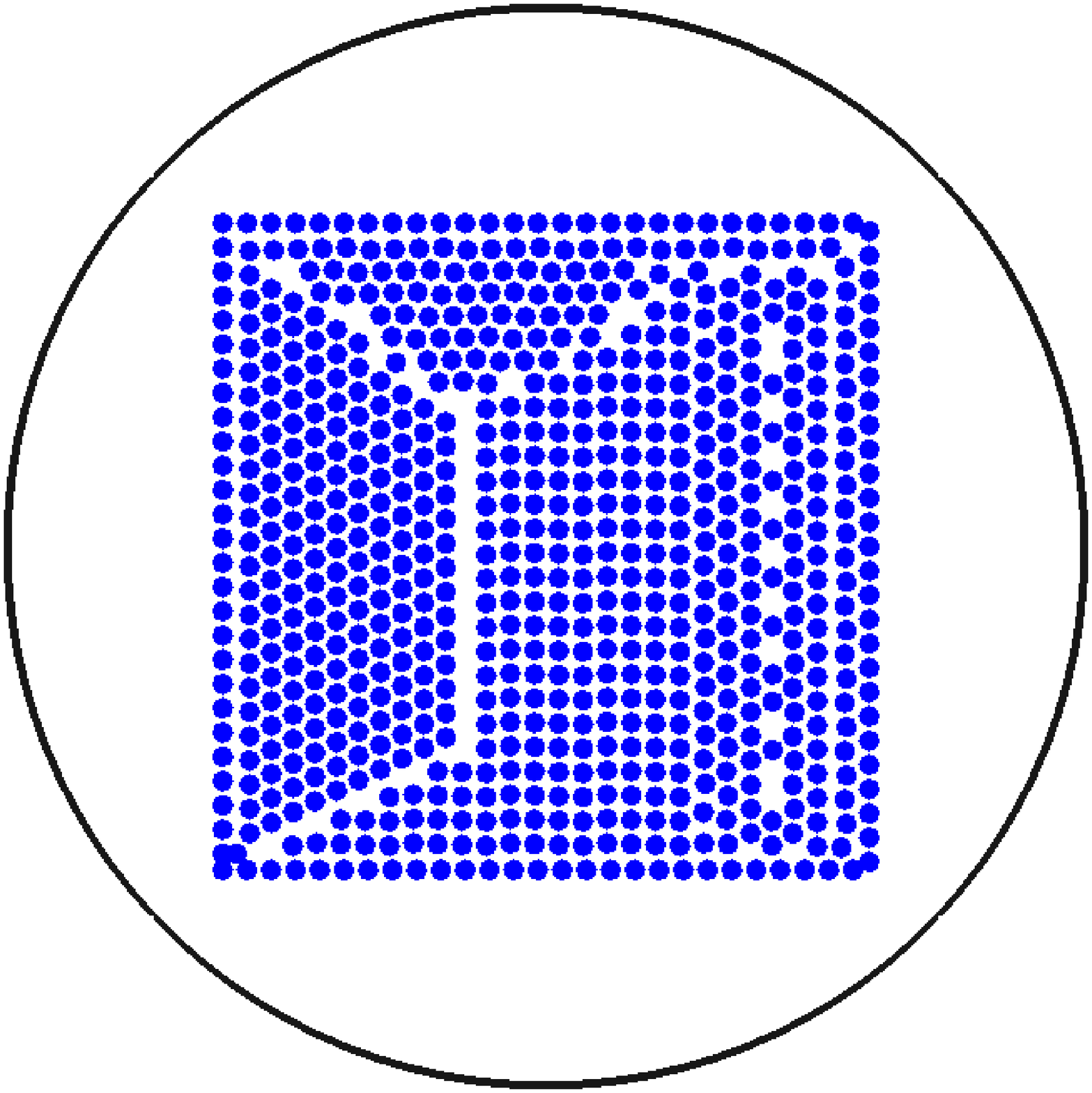}
 \includegraphics[width=52mm]{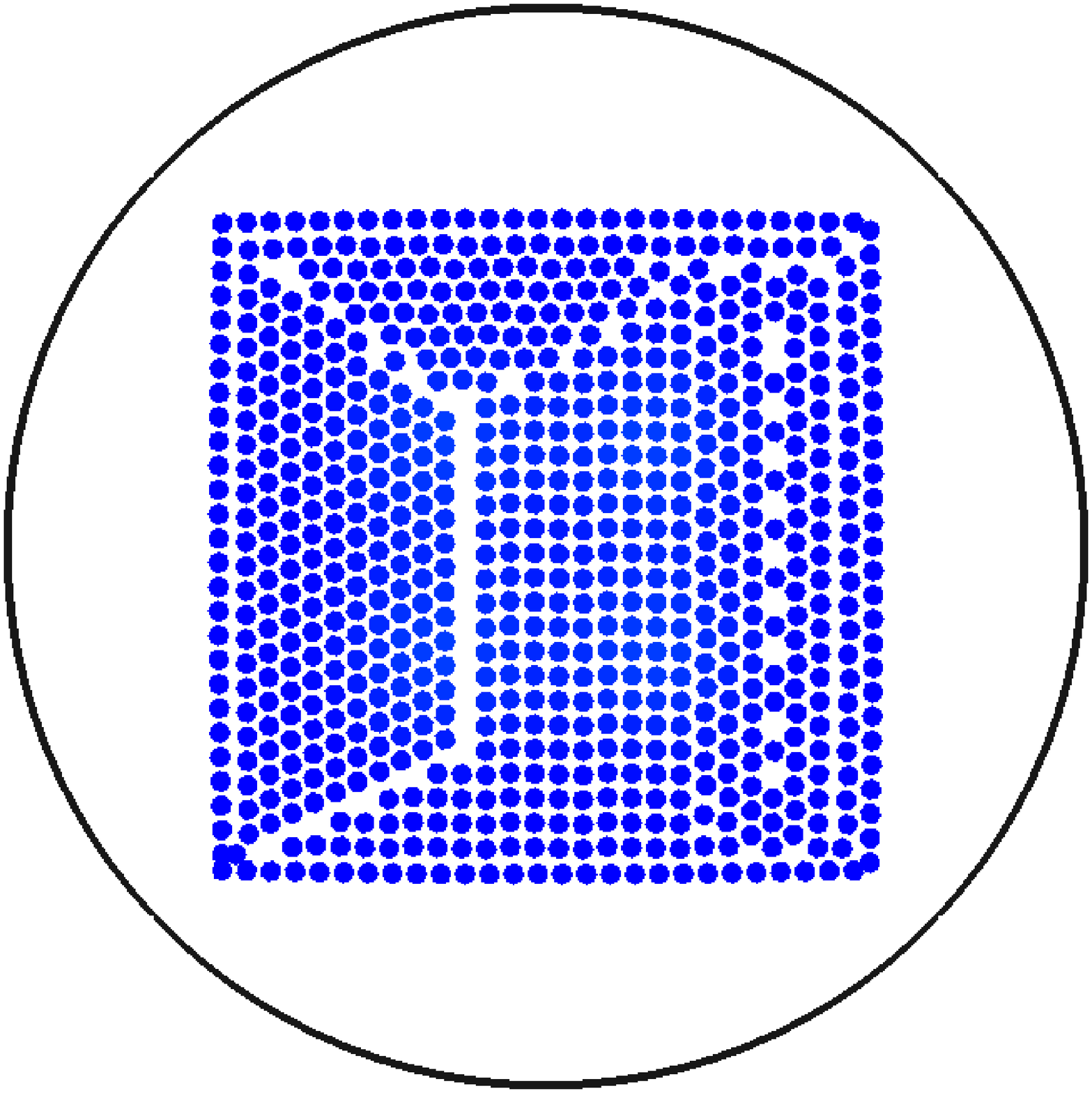}
 \includegraphics[width=52mm]{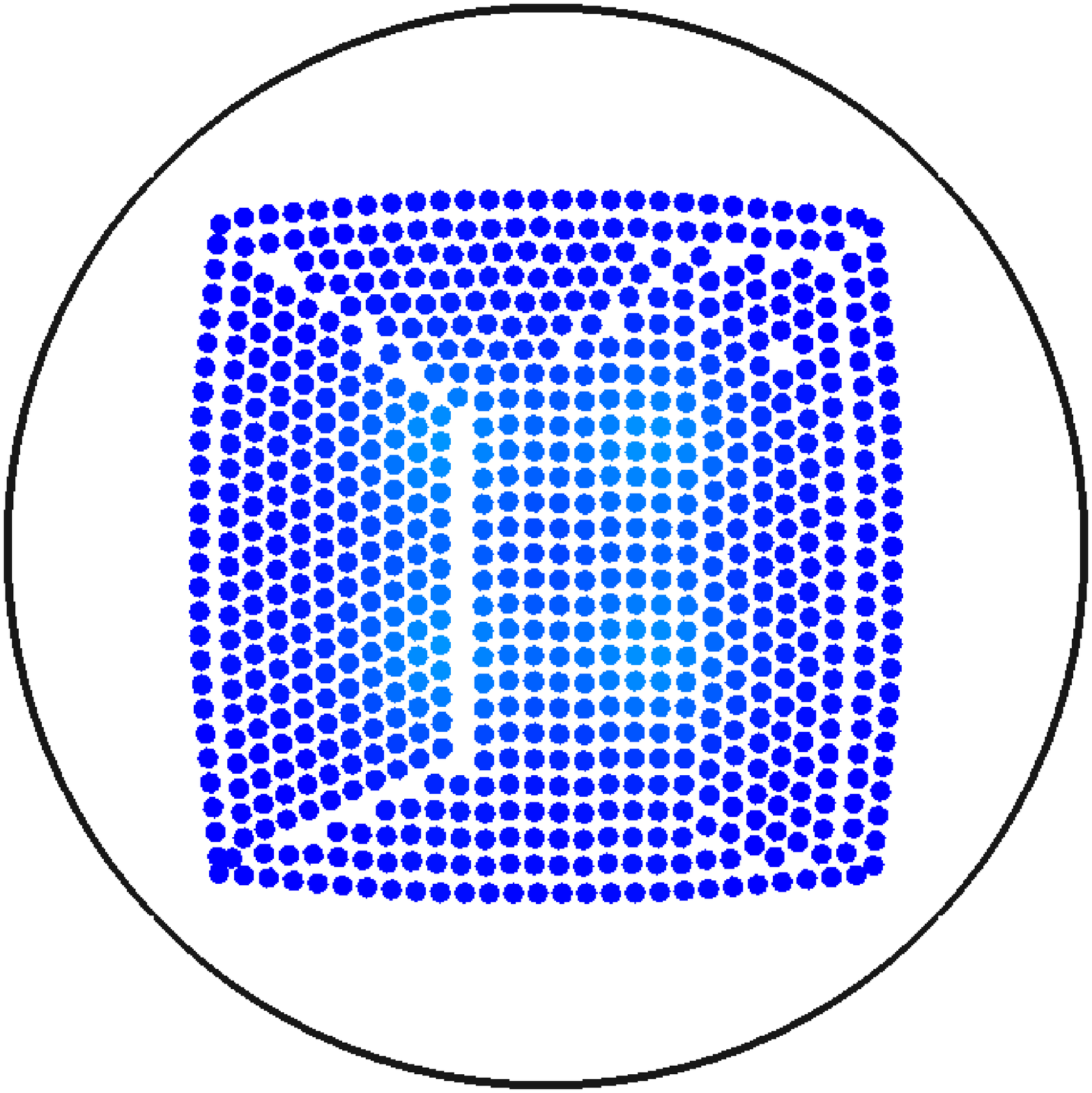} 
 \\[1ex] 
 \includegraphics[width=52mm]{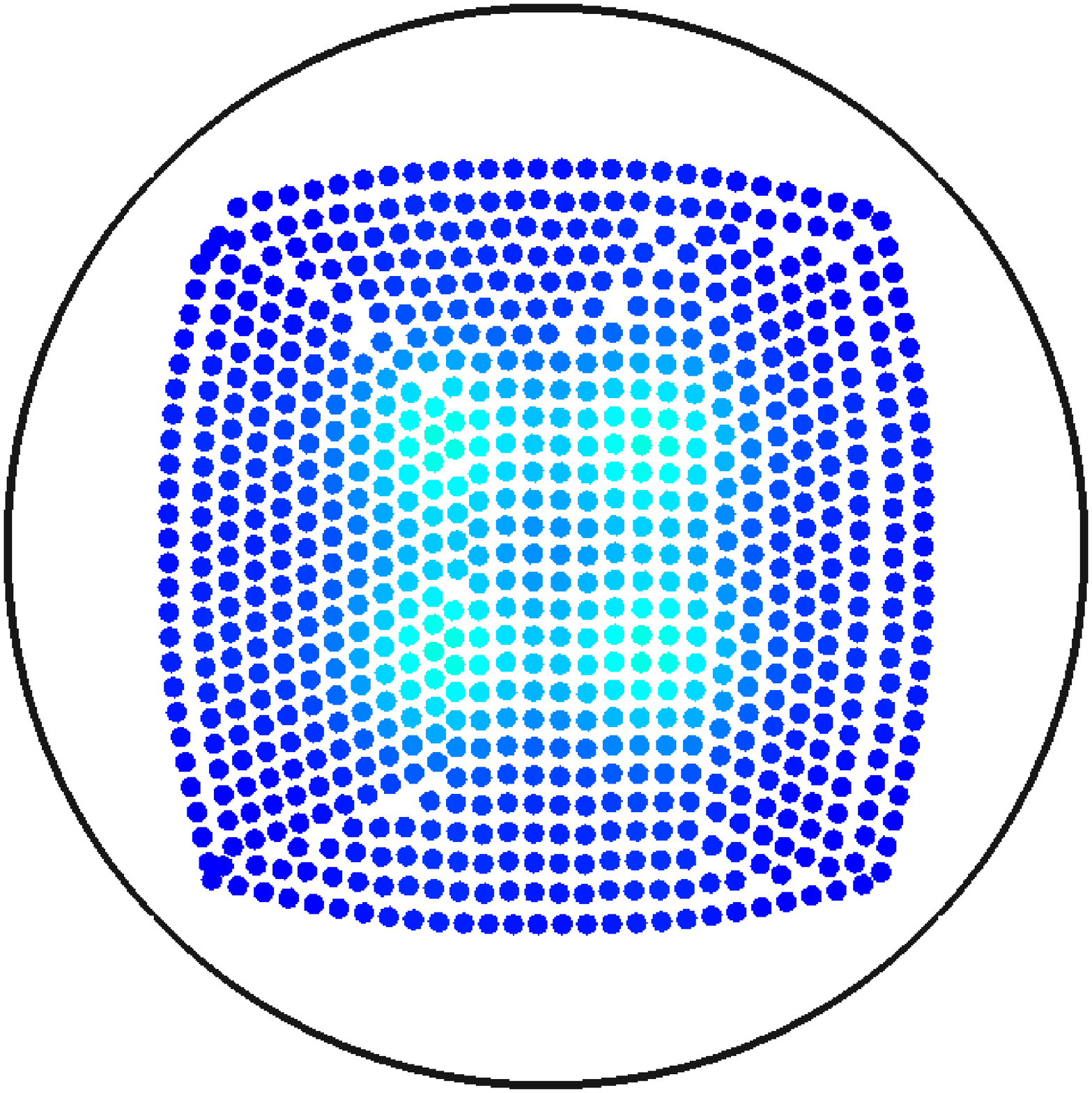}
 \includegraphics[width=52mm]{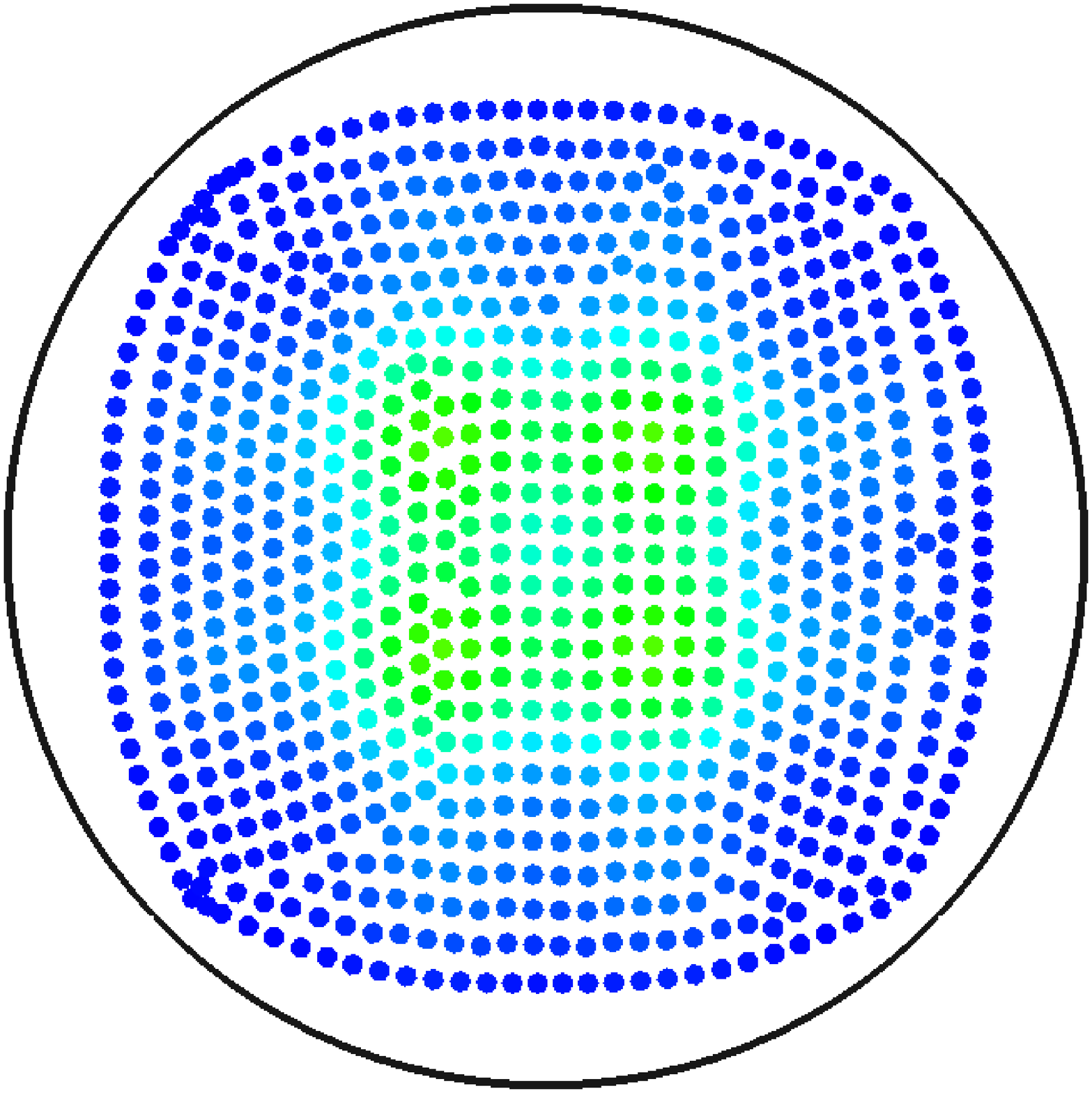}
 \includegraphics[width=52mm]{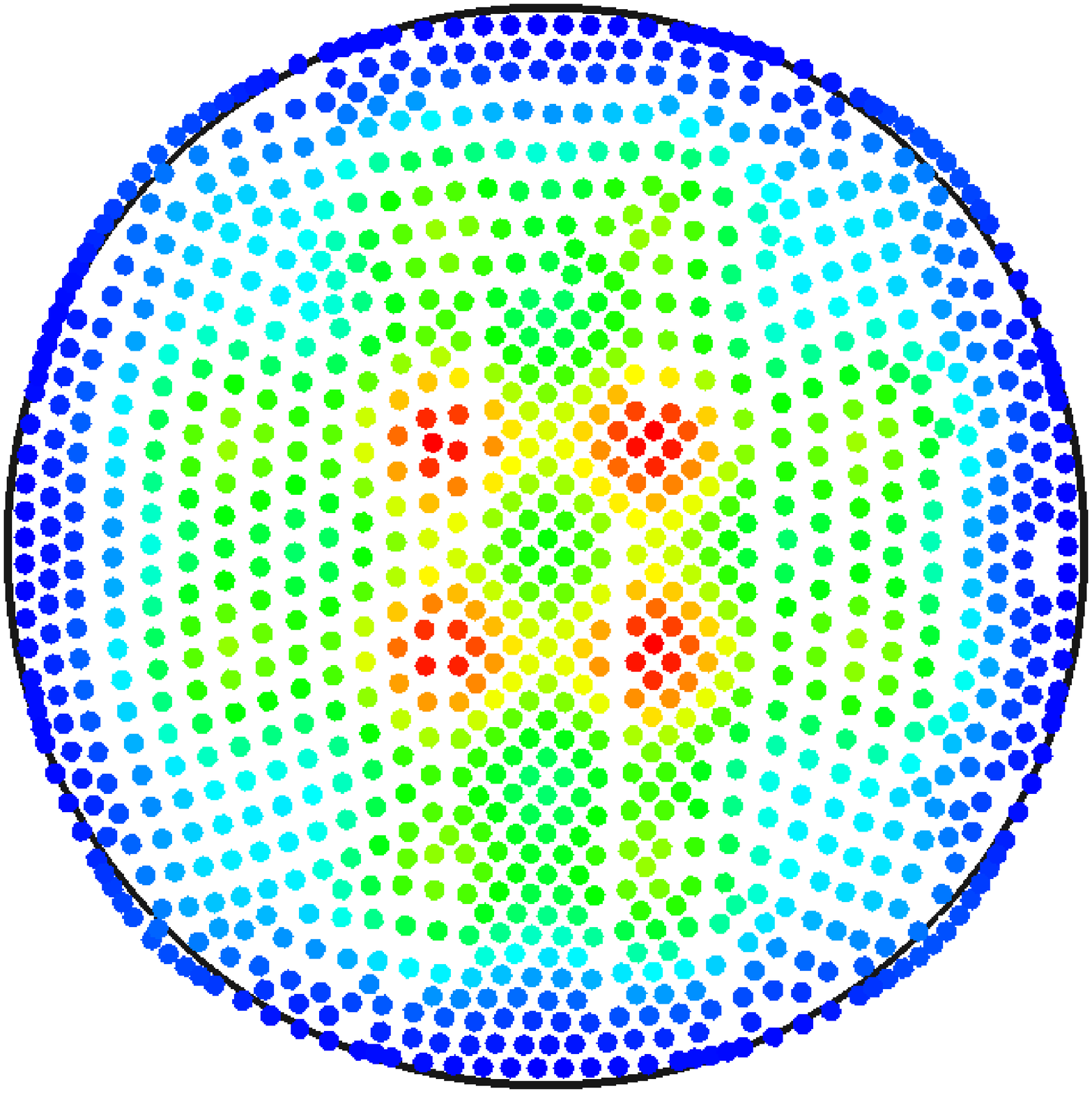}
 \caption[Optimal solution for $a=1.0,b=1.0$.]
 {Optimal solution for the Stefan problem with $a=1.0,b=1.0$. The time steps are equally spaces from $t=0$ to $T=0.3$.
 The colour represents the magnitude of the source and the black line the desired shape at final time.}
 \label{fig:stefanCircle}
\end{figure}

\begin{figure}[p]
 \centering
 \subfigure[Cost functional.]{
  \includegraphics[width=75mm]{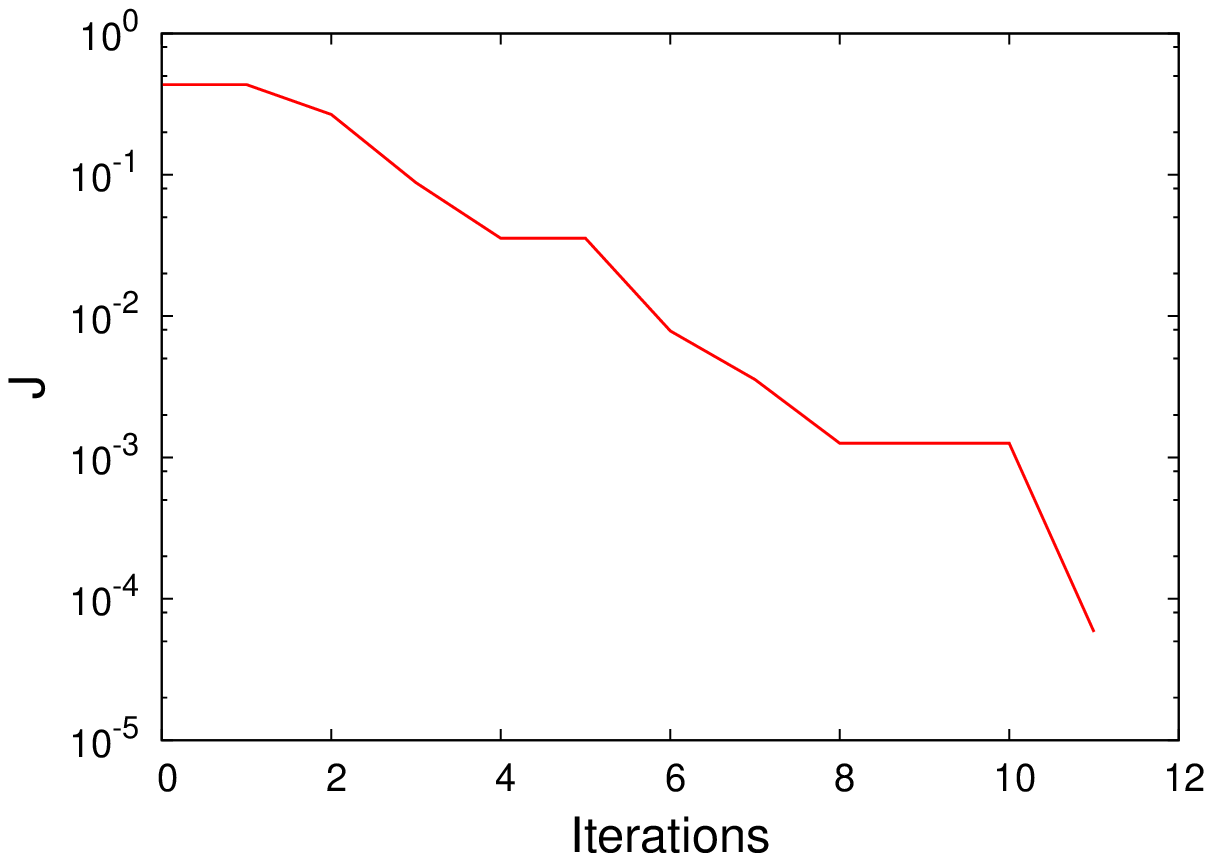}
 }
 \subfigure[Relative gradient norm.]{
  \includegraphics[width=75mm]{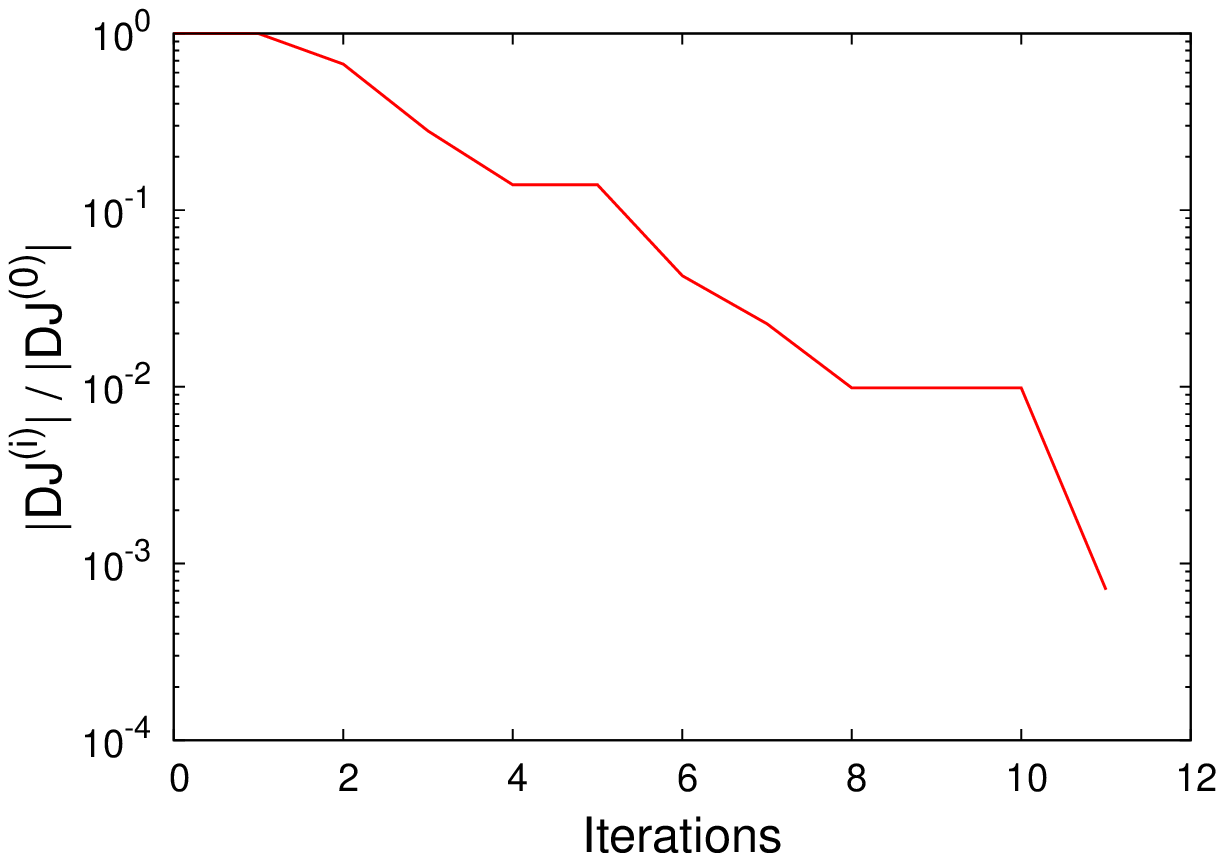}
 }
 \caption[Cost functional and gradient norm for $a=1.0,b=1.0$.]
  {Cost functional and gradient norm for $a=1.0,b=1.0$.}
 \label{fig:stefanCircleConv}
\end{figure}


\bibliography{Literature}

\end{document}